\documentclass[a4paper,10pt,oneside,notitlepage]{article}
\usepackage{times}
\usepackage{lmodern}
\usepackage[T1]{fontenc}
\usepackage[utf8]{inputenc}
\usepackage[english]{babel}
\usepackage{amssymb}
\usepackage{amsmath}
\usepackage{amsthm}
\usepackage{amsfonts}
\usepackage{latexsym}
\usepackage[pdftex]{color, graphicx}
\usepackage{makeidx}
\usepackage{sidecap}
\usepackage{verbatim}
\usepackage{geometry}
\usepackage{subfig}
\usepackage{mathrsfs} 
\usepackage{tikz}
\usetikzlibrary{shapes,arrows,shadows}
\usepackage{comment}

\definecolor{mygray}{RGB}{201,201,201}

\newtheorem{thm}{Theorem}[section]
\newtheorem{prop}[thm]{Proposition}
\newtheorem{lemma}[thm]{Lemma}

\newtheorem{remark}[thm]{Remark}

\DeclareMathOperator*{\dist}{dist}

\def\laplace{\Delta_x}
\def\uone{u_1^{\varepsilon}}
\def\uzero{u_0^{\varepsilon}}
\def\vone{v_1^{\varepsilon}}
\def\vzero{v_0^{\varepsilon}}
\def\normalder{\partial_{\nu}}
\def\normalderone{\partial_{\nu_1}}
\def\normalderzero{\partial_{\nu_0}}
\def\ballzero{\Omega_0}
\def\ballone{\Omega_1}
\def\boundone{\Gamma_1}
\def\boundzero{\Gamma_0}
\def\eigenv{\lambda^{\varepsilon}}
\def\grad{\nabla_x}
\def\newspace{\mathcal{H}_{\varepsilon}}
\def\visikop{\mathcal{K}_{\varepsilon}}
\def\almosteigenfun{\mathfrak{U}^{\varepsilon}}
\def\almosteigenval{\mathfrak{k}^{\varepsilon}}
\def\approxeigenv{\mathfrak{k}^{\varepsilon}}
\def\visikvec{\mathfrak{a}^{\varepsilon}}
\def\funcnewop{\mathfrak{u}^{\varepsilon}}
\def\column{\mathfrak{a}^{p\varepsilon}}
\def\columnq{\mathfrak{a}^{q\varepsilon}}
\def\cusp{\mathcal{O}}
\def\firstterm{\mathcal{U}_1(x_1,\eta)}

\def\remainder{\tilde{u}}
\def\U{\mathcal{U}}
\def\fu{\mathscr{U}}
\def\su{\mathscr{U}^\prime}

\title{The stiff Neumann problem: asymptotic specialty and \textquotedblleft kissing\textquotedblright domains}
\author{V. Chiad\`o Piat\thanks{Dipartimento di Scienze Matematiche, Politecnico di Torino
 Corso Duca degli Abruzzi 24,
10129 Torino, Italy {\tt e-mail:valeria.chiadopiat@polito.it, lorenza.delia@polito.it}}, L. D'Elia \footnotemark[1]
 \ and S.A. Nazarov \footnote{St. Petersburg State University, Universitetskaya nab., 7-9, St. Petersburg, 199034, Russia\hspace{0.2cm} \ and \hspace{0.2cm}
 Institute of Problems Mechanical Engineering RAS, V.O., Bolshoj pr., 61, St. Petersburg, 199178, Russia {\tt e-mail: srgnazarov@yahoo.co.uk}}
 }
\date{}

\begin{document}

\maketitle

\abstract
{We study the stiff spectral Neumann problem for the Laplace operator in a smooth bounded domain $\Omega\subset\mathbb{R}^d$ which is divided into two subdomains: an annulus $\ballone$ and a core $\ballzero$. The density and the stiffness constants are of order $\varepsilon^{-2m}$ and $\varepsilon^{-1}$  in $\ballzero$, while they are of order $1$ in $\ballone$. Here $m\in\mathbb{R}$ is fixed and $\varepsilon>0$ is small. We provide  asymptotics for the eigenvalues and the corresponding eigenfunctions as $\varepsilon \to 0$ for any $m$. 
%In the last part, the main difference of all previous of this and similar problem is that 
In dimension $2$ the case when $\ballzero$ touches the exterior boudary $\partial\Omega$ and $\ballone$ gets two cusps at a point $\cusp$ is included into consideration. The possibility to apply the same asymptotic procedure as in the \textquotedblleft smooth\textquotedblright\, case is based on the structure of eigenfunctions in the vicinity of the irregular part. The full asymptotic series as $x\to\cusp$ for solutions of the mixed boundary value problem for the Laplace operator in the cuspidal domain is given.
%We investigate the same stiff spectral problem in the case of \textquotedblleft kissing\textquotedblright domains $\ballzero$ and $\ballone$ in the plane $\mathbb{R}^2$, constructing the full asymptotic expansion, as $x\to\cusp$,  for the eigenfunctions of the mixed boundary value problem for the Laplace operator with the Neumann condition on the external boundary. Here $\cusp$ is the cuspidal point formed by two disks kissing each other from interior. 

%the asymptotic expansion of eigenfunctions of the Dirichlet-Neumann boundary value problem is presented. Here the subdomains $\ballone$ and $\ballzero$ are disks touching each other in a cuspidal point $\cusp$. 

\section{Introduction}
Let $\Omega$ be a smooth bounded domain in $\mathbb{R}^d$  and let $\ballone$ and $\ballzero$ be two bounded domains in $\mathbb{R}^d$ with smooth boundaries $\boundone$ and $\boundzero$ respectively such that $\partial\Omega = \boundone$, $\overline{\Omega}_0\subset\Omega$ and $\Omega = \ballzero \cup \ballone\cup \boundzero$. We refer to $\ballone$ as the annulus and $\ballzero$ as the core. A typical geometrical situation is drawn in Fig. \ref{fig:domain}, where the annulus is shaded.  We consider the spectral Neumann problem in $\Omega_1\cup\Omega_0$ with natural transmission conditions for a second order differential operator with piecewise constant coefficients
     \begin{align}
      -\laplace\uone(x) &= \eigenv\uone(x), \hspace{2.5cm} x\in\ballone, \label{uone}\\
      -\varepsilon^{-1}\laplace\uzero(x) &= \eigenv\varepsilon^{-2m}\uzero(x), \hspace{1.7cm} x\in\ballzero, \label{uzero}\\
      \normalderone\uone(x) &=0, \hspace{3.4cm} x\in \boundone, \label{Neumann}\\
      \uzero(x) = \uone(x), \qquad &  \varepsilon^{-1}\normalderzero\uzero(x) = \normalderzero\uone(x), \hspace{0.6cm} x\in \boundzero, \label{bduzero1}
    %  \varepsilon^{-1}\normalderzero\uzero(x) &= \normalderzero\uone(x), \hspace{1.3cm} x\in \boundzero,  
     \end{align}  
where $\normalderone$ and $\normalderzero$ denote the derivatives along outward and inward normal vectors $\nu_1$ and $\nu_0$ to  $\boundone$ and $\boundzero$ respectively, $\eigenv$ is the spectral parameter and $-2m\in\mathbb{R}$ a fixed exponent. From a physical point of view, the factor $\varepsilon^{-2m}$ reflects the dead-weight of the material, i.e. increasing $m$ makes the material heavier. In the first part of the paper, we discuss the asymptotic behavior as $\varepsilon\to 0$ of the eigenpairs $(\eigenv,u^{\varepsilon})$ of problem \eqref{uone}-\eqref{bduzero1}. We identify the (real) eigenfunctions $u^{\varepsilon}$ with the pairs of functions $\{\uzero, \uone\}$, where $u^{\varepsilon}_i$ stands for the restriction of $u^{\varepsilon}$ to $\Omega_i$, $i=0,1$.\\ %All values of parameter $m$ are investigated, putting a special emphasis on the case $0<m<1/2$. \\
Depending on the orders of the relative density, namely $\varepsilon^{-2m}$, we predict a different asymptotic behaviour of the eigenpairs $(\eigenv, u^{\varepsilon})$, as $\varepsilon\to 0$, putting a special emphasis on the case $0<m<1/2$.  We also characterize the eigenpairs of all the limit spectral problems which have a discrete spectrum. 
Both for $m\in (0,1/2)$ and $m\leq 0$, the limit problem is described by an eigenvalue problem for the Laplace operator in the annulus $\ballone$ with mixed boundary conditions, while in $\ballzero$ the zero-order term of asymptotics of eigenfunction $\uzero$ is a solution of the Laplace equation with Neumann condition. For $m>1/2$ an eigenvalue problem for the Laplace operator posed in the core $\ballzero$ characterizes the limit problem for $\varepsilon=0$, while the leading term  in the annulus $\ballone$ is a harmonic function. The case $m=1/2$ is investigated in the book \cite{SHSP89} in more general setting. Here we give an independent proof for the  reader's convenience. In this case the stiffness and the density constants are of the same order, i.e. $\varepsilon^{-1}$, in the equation \eqref{uzero}; two limit problems appear: the  spectral Neumann problem for the Laplace operator in $\ballzero$ and the spectral problem with the mixed Dirichlet-Neumann boundary conditions in $\ballone$. We derive estimates of convergerce rates in different situations.\\  
The spectral problems \eqref{uone}-\eqref{bduzero1} are of interest in many area of physiscs. For instance, they are considered in the study of reinforcement and elasticity problems (cf. \cite{AN00,AN99,AN96,P80}). In \cite{LNP05}, estimates of convergence rates of the spectrum of stiff elasticity problems are obtained. We also mention the papers \cite{GLNP06, GNP11}, where the authors deal with the asymptotics of spectral stiff problem in domains surrounded by a thin band depending on $\varepsilon$. For a study of asymptotics for vibrating systems containing a stiff region independent of the small parameter $\varepsilon$, we refer to Sections V.$7$-V.$10$ in \cite{SHSP89} and the papers \cite{GNEP19, LNP03, P03}. The problem considered in this paper arises also in the study of different properties of porous media. They are particularly treated in the homogenization theory (cf. \cite{ADH90,BCP04, BCP15, BCS16, P91}). In the context of second order differential operator with double periodic coefficients, we also mention \cite{BCNT17, BP18, HL00,HP03,Z05}, where the authors investigate how to give rise to  spectral gaps in the essential spectrum. \\
In the last part of the paper, we handle the same stiff problem \eqref{uone}-\eqref{bduzero1} but with a geometry of the domain $\Omega$, which differs from that drawn in Fig \ref{fig:domain}. A irregular point appears on the boundary $\partial\Omega$, consisting of the point $\cusp$ of tangency  of the two \textquotedblleft kissing\textquotedblright\;  disks $\ballzero$ and $\ballone$ in $\mathbb{R}^2$ (see Fig.\ref{fig:kissingdomains}). The main feature is that the ans\"{a}tze obtained when the boundary $\partial\Omega$ is smooth are still valid. It is worth to mention that in a certain sense the problem \eqref{uone}-\eqref{bduzero1} can be reduced to a regular perturbation in an operator setting depending on the exponent $m$. In this way the full asymptotic series for eigenpairs of the problem can be readily derived in the \textquotedblleft smooth\textquotedblright\, case after constructing the main asymptotic and first correction terms. However, in the case of \textquotedblleft kissing\textquotedblright\, domains the perturbation analysis becomes much more involved because of possible singularities of solutions at the irregular point $\cusp$. We succeed to prove that these singularities do not interest our asymptotic procedure in the Neumann stiff problem and explain why it does not work directly for the Dirichlet stiff problem, namely when the condiction \eqref{Neumann} is replaced by $\uone(x)=0$, $x\in\boundone$. Further investigation of Dirichlet stiff problem are left as open questions to be considered. We provide the asymptotic expansion as $x\to\cusp$ of the eigenfunction of the Laplace operator along with Neumann boundary condition on the exterior boundary $\boundone$ and a constant trace on the interior boundary $\boundzero$. The ansatz is made of particular functions depending on the geometry of the domain and the boundary conditions. Moreover, we show that all eigenfunctions decay exponentially as $x\to\cusp$ when we set a homogeneous Dirichlet condition on interior boundary $\boundzero$.

 %The Dirichlet case remains fully open. We discuss how to expand  the eigenfunctions of the limit problems appearing in all cases previously studied near the singularity.  We only investigate the case where the annulus  and the singularity consists of the point of tangency $\cusp$  For the dimension $d\geq 3$ the asymptotics take rather different form and will be subject of further research.

%For $m\leq0$ and $m\in (0,1/2)$, the eigenfunctions of the spectral limit problem are infinitely differentiable in the closure of $\overline{\ballone}$. If some irregularities appear on the boundary $\boundone$, the eigenfunctions may loose their smoothness. In the last part of the paper,

 %a study  with corner and conical points and edges. In the paper \cite{NT18}, the authors investigate the asymptotic behaviour of the eigenfunctions of Laplace operator along with Neumann boundary conditions in a bounded domain with a cuspidal point.\\
 %In this paper, we obtain the asymptotic formula for the eigenfunctions to the spectral problem \eqref{laptouchbal}-\eqref{bdtouchbal1} in a domain with irregular boundary. Here, for simplicity, the case of a single cuspidal point $\cusp$ is dealt. 
 %The eigenfunction has an asymptotic expansion  in $\cusp$ which is made of particular functions depending on the geometry of the domain and the boundary conditions. In particular, we will expand the eigenfunction when we impose a  We point out that if we   %and hence they are infinitely smooth on the closure of the domain $\boundzero$. 
\noindent We mention the paper \cite{NT18} in which the authors investigate the asymptotic behaviour of the eigenfunctions of Laplace operator along with Neumann boundary conditions in a bounded domain with a cuspidal point (cf \cite{N94, NST09, NT11}). The paper \cite{D96} discusses the regularity in the space of infinitely smooth functions in the case of cuspidal edges and the paper \cite{MNP82} investigates the regularity of solution of bi-harmonic operator in domains with cusps. We refer to the monographs \cite{KMR97, NP94} for a detailed study of elliptic boundary problems in domains with other type of singularities.

%we investigate the limit spectral problem \eqref{laptouchbal}-\eqref{bdtouchbal1} when some boundary irregularity occurs.  We derive the complete asymptotict series of eigenfunction.\

%The study of spectral properties of 

\noindent The paper is organized as follows. In Section $2$ we introduce the weak formulation of the problem \eqref{uone}-\eqref{bduzero1}. We deduce the formal asymptotic expansions for the eigenelements in the most interesting case $m\in(0,1/2)$, whose leading terms are determined by the constant $c_0$ (see \eqref{constant}) in $\ballzero$ and via Neumann spectral problem for Laplacian in $\ballone$. We also discuss briefly the infinite asymptotic series. Section $3$ contains the main result which is formulated in Theorem \ref{thmmain} and the justification of the asymptotics for $m\in (0,1/2)$. In Sections $4-7$ we present the asymptotic expansions of eigenelements for the remaining values of $m$. We introduce the problems which determine the leading and the first-order correction terms and we justify the expansions. In Section $8$ we derive and justify the asymptotic expansion of the eigenfunctions of the Laplace operator in $\ballone\setminus\cusp$ along with the homogenuous Neumann condition on $\boundone$ and the non-homogeneous Dirichlet boundary condition on $\boundzero$. Moreover we discuss some open questions. %Here $\cusp$ is the cuspidal point. 

 % We put a special emphasis for the case $0<m<1/2$ because of asymptotics of singular type. In the other case, the asymptotics for the eigenpairs is of regular types. In the very end, we study the limit problem in a different domain, the so-called kissing domain. Here, we study two disks in $\mathbb{R}^2$ in order to make easier computation. There are many questions related to this topis which are still open.\\ 
%Stiff problems are studied in the last decades and enjoys different physical interpretation. For example in \cite{LNP05} the authors study the convergence of spectrum and the estimates of convergence rates for eigen-pairs of the linearized elasticity equation. In \cite{BCNT17}, Bakharev et al. show that a suitable modification of coefficients of a scalar second order differential equation in the plane with double periodic coefficients cause an additional spectral band in the essential spectrum. They do not study the asymptotics and the structure is thin.
%Stiff problems are partial differential equations problems with very different values of coefficients in different parts of domain.   

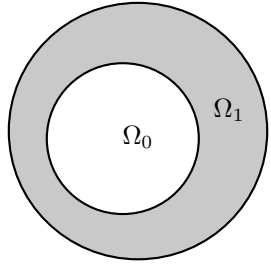
\begin{figure}
\centering
\begin{tikzpicture}
\draw [fill=mygray, thick] (2,0.5)circle(1.7)node[above, xshift=1.2cm]{$\Omega_1$};
\draw  [fill=white, thick](1.8,0.4) circle(1)node[below, xshift=0.2cm, yshift=0.3cm]{$\Omega_0$};
%\draw (0,0) arc (-90:-25:3cm) node[anchor=east] {$\Gamma_0$};
%\draw [<-] (0.9,0.1)--(2,-1) node[near end, below]{$\Omega$};
\end{tikzpicture}
\caption{Annulus domain $\Omega_1$ and core domain $\Omega_0$}\label{fig:domain}
\end{figure}
\vspace{0.5cm}

\section{Formal asymptotics in the case $\boldsymbol{0<m<1/2}$}

\subsection{Setting of the problem}
The variational formulation of  problem \eqref{uone}-\eqref{bduzero1} reads: find $\eigenv\in\mathbb{R}$ and $\{\uzero, \uone
\} \in H^1(\Omega)\setminus\{0\}$ satisfying
      \begin{equation}
      \label{weakform}
      (\grad \uone, \grad\varphi_1)_{\ballone} +\varepsilon^{-1}(\grad \uzero, \grad\varphi_0)_{\ballzero} = \eigenv ((\uone,\varphi_1)_{\ballone} + \varepsilon^{-2m}(\uzero,\varphi_0)_{\ballzero}) 
      \end{equation}
$\forall\varphi\in H^1(\Omega)$. Here, $(\cdot,\cdot)_{\Omega_i}$ denotes the natural inner product of Lesbegue space $L^2(\Omega_i)$, $i=0,1$ and $m\in\mathbb{R}$. % We identify a function $\varphi\in H^1(\Omega)$ with the pair of functions $\{\varphi_1,\varphi_0 \}$, with $\varphi_i$, $i=0, 1$ being the restriction to $\Omega_i$.\\ 
\noindent For each $\varepsilon>0$ the bilinear form on the left-hand side of \eqref{weakform} is positive, symmetric and closed in $H^1(\Omega)$. Due to comptacness of the embeddings $H^1(\Omega_i)\hookrightarrow L^2(\Omega_i), i=0,1$, the problem \eqref{uone}-\eqref{bduzero1} is associated with a self-adjoint operator whose spectrum consists of the monotone increasing unbounded sequence of eigenvalues (cf., for example, \cite[Theorems 10.1.5 and 10.2.2]{BS87})
       \begin{equation}
       \label{spectrum}
       0=\eigenv_1<\eigenv_2\leq\cdots\leq\eigenv_n\leq\cdots\rightarrow\infty
       \end{equation}
repeated according to their multiplicity. The corresponding eigenfunctions $\{ \uzero,\uone\}$ are subject to the orthonormalization conditions
     \begin{equation}
     \label{normalcond}
     ( u^{\varepsilon}_{1,i}, u^{\varepsilon}_{1,j})_{\ballone} + \varepsilon^{-2m}( u^{\varepsilon}_{0,i}, u^{\varepsilon}_{0,j})_{\ballzero}=\delta_{i,j}, \qquad i,j\in\mathbb{N},
      \end{equation} 
where $\delta_{i,j}$ is the Kronecker symbol. The orthonormalization condition \eqref{normalcond} suggests to perform the replacements
     \begin{align}
     \vone(x) = \uone(x),  \hspace{0.3cm}x\in\ballone,\hspace{1.3cm}
     \vzero(x)= \varepsilon^{-m}\uzero(x), \hspace{0.3cm}x\in\ballzero. \label{vzero}
     \end{align}
Hence, $\{\vzero,\vone\}$ satify the orthonormalization condition in $L^2(\Omega)$ which does not depend anymore on $\varepsilon$.
 %  \begin{equation*}
  %   \|\vone\|^2_{L^2(\ballone)}+\|\vzero\|^2_{L^2(\ballzero)} = 1,
   %\end{equation*}
Equations  \eqref{uone}-\eqref{uzero} remain unchanged, while the transmission conditions \eqref{bduzero1} turn into
       \begin{align*}
       \varepsilon^{m}\vzero(x) = \vone (x), \hspace{1cm}
       \varepsilon^{m-1}\normalderzero\vzero(x) = \normalderone\vone(x), \hspace{0.4cm}x\in\boundzero.
       \end{align*}
We look for the asymptotic expansion of eigenfunctions $\{\vzero, \vone\}$ in the form 
    \begin{align}
    \vzero(x) &= \varepsilon^{m}v_{0}^0(x)  + \varepsilon^{1-m} v^\prime_0(x) + \cdots,\hspace{0.5cm} x\in\ballzero,\label{asyu_0}\\
    \vone(x)&= v^{0}_1(x)+\varepsilon^{2m}v^\prime_1(x)+\cdots,\hspace{1.1cm}x\in\ballone\label{asyu_1}.
    \end{align}
%We impose 
 %    \begin{align*}
  %   \|v^0_1\|_{L^2(\ballone)} &= 1, \hspace{1cm} (v^0_1, v^\prime_1)_{\ballone} =0,\\
   %  \overline{v}^\prime_0 &= \frac{1}{|\ballzero|}\int_{\ballzero} v^\prime _0(x)dx=0.
    % \end{align*}
We assume that the eigenvalue $\eigenv$ admits the asymptotic ansatz
     \begin{equation}
     \label{asyeigen}
     \eigenv = \lambda^{0} + \varepsilon^{2m}\lambda^\prime+\cdots.
     \end{equation}
By inserting expansions \eqref{asyu_0}, \eqref{asyu_1}, \eqref{asyeigen} in the spectral problem \eqref{uone}-\eqref{bduzero1}, we collect coefficients of the alike powers of $\varepsilon$ and gather boundary value problems for $v^0_0, v^\prime_0$ and $v^0_1, v^\prime_1$.

\subsection{Problem for $\boldsymbol{v_0^0}$ and $\boldsymbol{v^\prime_0}$}
The leading term in \eqref{asyu_0} is a solution of the problem% the Laplace equation with homogeneuos Neumann condition
     \begin{align}
     -\laplace v_0^0(x) =0, \hspace{0.3cm}x\in\ballzero,\hspace{1cm}\normalderzero v_0^0(x) =0, \hspace{0.3cm}x\in\boundzero,\label{pbv_0^0}     
     \end{align}
and hence $v^0_0=c_0$. At this stage, $c_0$ is an arbitrary constant in $\mathbb{R}$. The first-order correction term in \eqref{asyu_0} satisfies the boundary value problem 
        \begin{align}
        -\laplace v^\prime_0(x) = \lambda^0v_0^0(x), \hspace{0.3cm}x\in\ballzero,\hspace{1cm} \normalderzero v^\prime_0 (x)&= \normalderzero v^0_1(x), \hspace{0.3cm}x\in\boundzero. \label{pbv'_0}
        \end{align}
From the compatibility condition for inhomogeneous Neumann problem, we determine the constant $c_0$:
     \begin{equation}
     \label{constant}
     c_0 = \frac{1}{\lambda^0|\ballzero|}\int_{\boundzero} \normalderzero v^0_1 ds_x, 
     \end{equation}
where $|\cdot|$ stands for Lebesgue measure of a set and $\lambda^0 \neq 0$ is an eigenvalue of the problem \eqref{probv_1^0}-\eqref{bcv_1^01}.

\subsection{Problem for $\boldsymbol{v^0_1}$ and $\boldsymbol{v^\prime_1}$}
The leading terms  in \eqref{asyu_1} and \eqref{asyeigen}  verify the spectral mixed boundary value problem 
      \begin{align}
      -\laplace v^0_1(x) &= \lambda^0v^0_1(x),\hspace{0.5cm}x\in \ballone, \label{probv_1^0}\\
      \normalderone v^0_1(x) =0, \hspace{0.3cm}x\in \boundone, &\hspace{1cm}
      v^0_1(x)=0,\hspace{0.3cm}x\in\boundzero. \label{bcv_1^01}
      \end{align}
The variational setting implies the integral identity 
     \begin{equation*}
      (\grad v^0_1, \grad\varphi)_{\ballone} = \lambda^0(v^0_1, \varphi)_{\ballone},\qquad \varphi \in H^1_0(\ballone, \boundzero),
     \end{equation*}
where $H^1_0(\ballone, \boundzero):=\{u\in H^1(\ballone)\hspace{0.01cm}:\hspace{0.01cm} u_{ |\boundzero} =0 \}$.   The spectrum of problem \eqref{probv_1^0}-\eqref{bcv_1^01} is discrete and turns into a monotone unbounded sequences of eigenvalues
      \begin{equation}
      \label{seqeigen}
       0<\lambda^0_1 < \lambda^0_2 \leq \dots \leq \lambda^0_n\leq\cdots \rightarrow+\infty,
      \end{equation}
and the corresponding eigenfunctions $v^0_{1,1}, v^0_{1,2}, \dots $ are subject to the orthonormalization conditions
     \begin{equation}
     \label{orthnormv^0_1}
     (v^0_{1,i}, v^0_{1,j})_{\ballone} =\delta_{i,j}, \qquad i,j\in\mathbb{N}.
     \end{equation}
The correction term in \eqref{asyu_1} is determined by the boundary value problem
      \begin{align}
      -\laplace v^\prime_1(x) - \lambda^0 v^\prime_1(x) &= \lambda^\prime v^0_1(x), \hspace{0.5cm}x\in\ballone,\label{pbv'_1}\\
      \normalderone v^\prime_1(x) =0, \hspace{0.3cm}x\in\boundone,\hspace{1cm}&\hspace{1cm}
      v^\prime_1(x) = v^0_0(x),\hspace{0.3cm} x\in\boundzero \label{bcv'}.
      \end{align}
Since $v^0_0=c_0$ is fixed and defined by \eqref{constant}, the boundary condition \eqref{bcv'} becomes $v^\prime_1(x) = c_0$, $x\in\boundzero$.
 %  \begin{equation*}
 %  \label{bcv'1}
 %  v^\prime_1(x) = c_0,\hspace{0.8cm} x\in\boundzero.
 %  \end{equation*}
The correction term $\lambda^\prime$ is determined through the compatibility condition in the problem \eqref{pbv'_1}-\eqref{bcv'}. First, 
we assume that the eigenvalue $\lambda^0_n\neq 0$ of problem \eqref{probv_1^0}-\eqref{bcv_1^01} is simple. Then the problem \eqref{pbv'_1}-\eqref{bcv'} has a unique solution if and only if   
      \begin{align*}
      \lambda^\prime_n\int_{\ballone}|v^0_{1,n}(x)|^2dx = % c_0\int_{\boundzero}\normalderzero v^0_{1,n}(x)ds_x =
       c_0\int_{\boundzero}\normalderzero v^\prime_{0,n}(x)ds_x
       =-c_0\int_{\ballzero}\laplace v^\prime_{0,n}(x) dx = %c_0\int_{\ballzero} \lambda^0_{n}c_0dx =
       c^2_0\lambda^0_n|\ballzero|.
      \end{align*}
Thus, the perturbation term in the ansatz \eqref{asyeigen} takes the form
     \begin{equation}
     \label{singleeigenv}
     \lambda^\prime_n = c_0^2\lambda^0_n|\ballzero| = \frac{1}{\lambda^0_n|\ballzero|}\left(\int_{\boundzero} \normalderzero v^0_1 ds_x\right)^2.
     \end{equation}

\subsubsection{Multiple eingenvalues}
In the case $\lambda^0_n\neq 0$ is a multiple eigenvalue with multiplicity $\tau>1$, i.e.
    \begin{equation}
    \label{multipleig}
    \lambda^0_{n-1}< \lambda^0_n  = \lambda^0_{n+1}= \dots = \lambda^0_{n+\tau-1}<\lambda^0_{n+\tau},
    \end{equation} 
the expansions \eqref{asyu_0}-\eqref{asyu_1} are still valid. However we predict that the leading terms of $v^{\varepsilon}_{1,n}, v^{\varepsilon}_{1,n+1}, \dots,$ $ v^{\varepsilon}_{1,n+\tau-1}$ are linear combinations of the eigenfunctions $v^0_{1,n}, v^0_{1,n+1}, \dots, v^0_{1,n+\tau-1}$ of the problem \eqref{probv_1^0}-\eqref{bcv_1^01} associated to eigenvalue $\lambda^0_n$, i.e.
      \begin{equation}
      \label{lincom}
      V_{1,j}^{0}(x) = a_n^jv^0_{1,n}(x)+\dots+a_{n+\tau -1}^jv^{0}_{1,n+\tau-1}(x), \qquad j= n,\dots, n+\tau-1.
     \end{equation}
Furthermore, we require that the columns
    \begin{equation*}
     a^j = (a^j_n, \dots, a^j_{n+\tau-1})^\top\in\mathbb{R}^{\tau}, \qquad j=n,\dots,n+\tau-1,
    \end{equation*} 
satisfy the orthonormalization conditions 
      \begin{equation}
      \label{orthcondeigenv}
      (a^j, a^i) := \sum_{k=n}^{n+\tau-1}a^j_k a^i_k = \delta_{j,i}, \qquad j,i=n,\dots,n+\tau-1.
      \end{equation}
As a consequence, the linear combinations \eqref{lincom} with $j=n,\dots,n+\tau-1$ are a new orthonormal basis in the eigenspace of the eigenvalue $\lambda^0_n$.\\ Bearing in mind the linear combinations \eqref{lincom}, the compatibility conditions in the problem \eqref{pbv'_0} yield  the new constant leading terms $v^0_{0,n}, \dots, v^0_{0,n+\tau-1} $ of the ansatz \eqref{asyu_0}
   \begin{equation}
   \label{c_{0,k}}
   v^0_{0,j} = \frac{1}{\lambda^0_n|\ballzero|}\sum_{k=n}^{n+\tau-1}a^j_k\int_{\boundzero} \normalderzero v^0_{1,k}ds_x, \qquad j=1,\dots,n+\tau-1.
   \end{equation}
     
%and $\tau$ copies of the correcting term in the eigenvalue ansatz \eqref{asyeigen}
 %     \begin{equation*}
  %    \lambda'_n, \lambda'_{n+1}, \cdots, \lambda'_{n+\tau-1}.
   %   \end{equation*}
%The functions $\{v^0_k \}$, $k=n,\cdots, n+\tau-1$ in the linear combination \eqref{lincom} are the eigenfunctions corresponding to the eigenvalue $\lambda^0_n$ of the problem \eqref{probv_1^0}-\eqref{bcv_1^0}. We require that the columns
     
%satisfy the orthogonality and normalization condition
\noindent The correction term $V^\prime_{1,j}$ is determined from the problem 
      \begin{align}
      -\laplace V^\prime_{1,j}(x) -\lambda^0_n V^\prime_{1,j}(x) &= \lambda^\prime_jV^{0}_{1,j}(x), \hspace{0.8cm}x\in\ballone, \label{multiple}\\
      \normalderone V^\prime_{1,j}(x)=0, \hspace{0.3cm}x\in \boundone,\hspace{0.6cm}&\hspace{0.7cm}
      V^\prime_{1,j} (x) = v^0_{0,j}, \hspace{0.3cm}x\in\boundzero.\label{multiple2}
      \end{align}
%where
 %     \begin{equation}
  %    \label{c_{0,k}}
   %   c_{0,k} = \frac{1}{\lambda^0_n|\ballzero|}\int_{\boundzero} \normalderzero v^0_{1,k}(x) ds_x, \qquad k=n,\cdots, n+\tau-1.
    %  \end{equation}
The Fredholm alternative  leading to the necessary and sufficient condition for $V^\prime_{1,j}, j=n,\dots,$ $ n+\tau-1$ to exist, is given by  
       \begin{equation*}
       \lambda^\prime_j (V^{0}_{1,j}, v^0_{1,p})_{\ballone} =\int_{\boundzero}V^\prime_{1,j} \normalderzero v^0_{1,p}(x) ds_x, \qquad p=n,\dots,n+\tau-1.
      \end{equation*}
Owing to \eqref{c_{0,k}} and  the orthonormalization condition \eqref{orthnormv^0_1}, the above formulas become
        \begin{equation}
        \label{multeig}
        \lambda^\prime_j a^j_p = \sum_{k=n}^{n+\tau-1} a^j_k\frac{1}{\lambda^0_n|\ballzero|}\int_{\boundzero} \normalderzero v^0_{1,k}(x) ds_x \int_{\boundzero}\normalderzero v^0_{1,p}(x) ds_x, \qquad p=n\dots,n+\tau-1.
       \end{equation}
We represent the relations \eqref{multeig} as an algebraic spectral system
     \begin{equation*}
      \mathcal{M}a^j= \lambda^\prime_ja^j, \qquad j=n,\dots, n+\tau-1,
     \end{equation*}
with the matrix $\mathcal{M}$ of size $\tau\times\tau$  defined by % $M^j(x)=(M^j_{p,k}(x))_{p,k=n}^{n+\tau-1}$ are given by 
      \begin{equation*}
      \label{matrix}
      \mathcal{M}_{pk} = \frac{1}{\lambda^0_n|\ballzero|} \int_{\boundzero}\normalderzero v^0_{1,p}(x)ds_x \int_{\boundzero}\normalderzero v^0_{1,k}(x)ds_x, \qquad p,k=n,\dots,n+\tau-1.
      \end{equation*}
It is clear that $\mathcal{M}$ is a symmetric matrix, i.e. $\mathcal{M}_{pk} = \mathcal{M}_{kp} $. Therefore, it has $\tau$ real eigenvalues, $\lambda^\prime_{n}, \lambda^\prime_{n+1}, \dots, \lambda^\prime_{n+\tau-1}$, with eigenvectors $a^n, a^{n+1}, \dots, a^{n+\tau-1}$ satisfying the orthonormalization conditions \eqref{orthcondeigenv}.
Since the determinant of the matrix $\mathcal{M}$ and all its minors of order $k,$ $1\leq k\leq \tau-1$, are equal to $0$, the characteristic polynomial of $\mathcal{M}$ is simply
    \begin{equation}
    \label{polcar}
    (\lambda^\prime)^{\tau} - tr(\mathcal{M})(\lambda^\prime)^{\tau-1} = 0,
    \end{equation} 
with $tr(\mathcal{M})$ being the trace of the matrix $\mathcal{M}$. It follows that the roots of \eqref{polcar} $\lambda^\prime_{j}$ , $ j=n,\dots, n+\tau-1$, are given by
    \begin{equation}
    \label{multipleeigenv}
    \lambda^\prime_{n}=\dots=\lambda^\prime_{n+\tau-2} = 0, \qquad \lambda^\prime_{n+\tau-1} = tr(\mathcal{M}) =\frac{1}{\lambda^0_n|\ballzero|}\sum_{k=n}^{n+\tau-1}\left(\int_{\boundzero} \normalderzero v^0_{1,k}(x)ds_x\right)^2.
    \end{equation}

\subsection{Final remarks}
The asymptotic procedure described above can be continued to construct infinite asymptotic series for eigenvalues and eigenfunctions of the problem \eqref{uone}-\eqref{bduzero1}. If the eigenvalues $\lambda_n^0$ is simple, the analysis just repeats the explained steps and provides the formal series
    \begin{equation}
    \label{seri}
    \sum_{j,k=0}^{\infty} \varepsilon^{jm+k(1-2m)}\lambda^{(j,k)}_n,
    \end{equation}
and the difference between the true eigenvalue $\eigenv$ and the partial sum of the series \eqref{seri} can be estimated in a way, quite similar to Section $3$. \\ The same can be readily done in the case $\tau=2$ when the correction term $\lambda^\prime_{n+1}$ in \eqref{multipleeigenv} does not vanish so that both the eigenvalues $\eigenv_n$ and $\eigenv_{n+1}$ become simple and therefore can be examined independently. However, if $\lambda^0_n$ has multiplicity $\tau>2$ or $\tau=2$ with $\lambda^\prime_{n+1}=0$ (cf. \eqref{multipleeigenv}), the coefficients of the linear combination \eqref{lincom} are not completely determined. In order to compute them,  the coefficients $a^j_n, \dots, a^j_{n+\tau-1}$ are assumed to be a linear combination of the eigencolumns associated to the eigenvalue $0$ of the matrix $\mathcal{M}$, obtaining the coefficients and the next term of the expansion of $\eigenv$. Nevertheless, there is no argument ensuring that the new matrix has distint eigenvalues and hence the coefficients of linear combination of $a^j_n, \dots, a^j_{n+\tau-1}$ can not be uniquely defined, so that an iteration of the previous procedure is needed again.

%We conclude with some comments about the infinite series of $\eigenv, \uzero, \uone$. In the case of simple eigenvalue $\lambda^0_n$, the asympotic expansions \eqref{asyu_0}-\eqref{asyeigen} have powers $\varepsilon^{\beta_{jk}}$, where $\beta_{jk} = jm+k(1-2m)$, $k,j=0,1,2,\dots$. In the case of eigenvalues $\lambda^0_n$ with mulitplicity $\tau=2$ and the first-order correction term $\lambda^\prime_n>0$, we find uniquely the coefficients $a^j = (a^j_n, a^j_{n+1})$ in the linear combination \eqref{lincom}. If $\lambda^0_n$ has multiplicity $\tau>2$ or $\tau=2$ with $\lambda^\prime_n =0$, then 

\section{Main result}
We present the main result of this paper, which is valid for any value $m\in\mathbb{R}$.
      \begin{thm}
        \label{thmmain}
        For $m\in \mathbb{R}$ and for any $N\in\mathbb{N}$ there exist  $\varepsilon_{N,m}>0$ and  $C_{N,m}>0$ such that the estimate
                \begin{equation}
                 \label{justasy}
                |\eigenv_n-\varepsilon^{\alpha}\lambda^0_n-\varepsilon^{\beta}\lambda^\prime_n|\leq C_{N,m}\varepsilon^{\gamma}, \qquad n=1,\dots, N, 
                \end{equation}
        holds for some $\alpha, \beta$ and $\gamma$, depending only on $m$, and $\varepsilon\in (0,\varepsilon_{N,m})$. 
        \end{thm}
   \begin{remark}
   \label{remarkmainthm} In the estimate \eqref{justasy}, $\eigenv_n$ is the $n$-th eigenvalue of the problem \eqref{uone}-\eqref{bduzero1}, $\lambda^0_n$ and $\lambda^\prime_n$ are the corresponding leading and first-order correction terms, appearing in the different ans\"{a}tze for $\eigenv_n$, which we will define in the forthcoming sections.
   \end{remark}
\noindent In the next subsection we provide the proof of the Theorem \ref{thmmain} with $m\in (0,1/2)$, where $\alpha=0$, $\beta=2m$, $\gamma = \min\{3m, 1\}$, and $\lambda^\prime_n$ is given by formula \eqref{singleeigenv}  for a simple eigenvalues and formulas \eqref{multipleeigenv} for multiple ones. The proof is split in two steps. The first one consists in proving partially that the eigenpairs $(\eigenv, \{\uzero,\uone\})$  converge to $(\lambda^0, \{0,u^0_1\})$, where $(\lambda^0, u^0_1)$ is an eigenpair of the limit problem \eqref{probv_1^0}-\eqref{bcv_1^01}. In the second step, we will use the so-called Lemma about near eigenvalues and eigenfunctions (cf.\cite{VL57}) in order to conclude with the proof of the Theorem \ref{thmmain}.

\subsection{Justification of asymptotics in the case \\ $\boldsymbol{m\in(0,1/2)}$}

%In this Section, we show this result only in the case $m\in (0,1/2)$, where $\gamma = \min\{3m, 1\}$ and $\beta=2m$ and the correction term  $\lambda^\prime_n$ is given by  the formula \eqref{singleeigenv} for a simple eigenvalue and formulas \eqref{multipleeigenv} for multiple ones. The proof of the Theorem \ref{thmmain} is split in two steps. The first one consists in proving partially that the eigenpairs $(\eigenv, \{\uzero,\uone\})$ converge to $(\lambda^0, \{u^0_1, 0\})$, where $(\lambda^0, u^0_1)$ is the eigenpair of the limit problem \eqref{probv_1^0}-\eqref{bcv_1^01}. In the second step, we will use the so-called Lemma about near eigenvalues and eigenfunctions (cf. \cite{VL57}) in order to conclude with the proof of the Theorem \ref{thmmain}.
\subsubsection{Step 1: Convergence theorem}
In this subsection, we show that for fixed $n\in\mathbb{N}$ the eigenvalue $\eigenv_n$ converges to $\lambda^0_n$, as $\varepsilon\to 0$, and the corresponding eigenfunctions converge strongly in $L^2(\ballone)$.  
       
%In this subsection, we provide a primitive result on convergence of eigenpairs of the problem \eqref{uone}-\eqref{bduzero2}, whose proof will be finished in the next subsection. 
 %\eqref{asyu_0}, \eqref{asyu_1}, \eqref{asyeigen}. %Throughout this section, we use the spectral problems \eqref{uone}-\eqref{bduzero2} involving eigenfunctions $\uone, \uzero$.
%Fixed $n\in\mathbb{N}$, we will prove that $\eigenv_n\rightarrow \lambda^0_n$, $u^{\epsilon}_{1,n}\rightarrow u^{0}_{1,n}$ and $u^{\epsilon}_{0,n}\rightarrow u^{0}_{0,n}$ as $\epsilon\rightarrow 0$. 
%In this and the next section, we aim to prove the following result.

    \begin{prop}
    \label{prop}
    The eigenvalues $\eigenv_n$ of the problem \eqref{uone}-\eqref{bduzero1} and the eigenvalues $\lambda^0_n$ of the problem \eqref{probv_1^0}-\eqref{bcv_1^01} are related by passing to the limit
         \begin{equation*}
          \eigenv_n \to\lambda^0_n, \qquad \text{as}\quad \varepsilon\to 0,\qquad n\in\mathbb{N}.
         \end{equation*}
    \end{prop}
We begin to show the following lemma.
    \begin{lemma}
    \label{lemma}
    Assume that for any $n\in\mathbb{N}$ there exist  $\varepsilon_n>0$ and $C_n>0$ such that
           \begin{equation}
           \label{estieigen}
           0<\eigenv_n\leq C_n \qquad \text{for}\hspace{0.2cm} \varepsilon\in (0, \varepsilon_n).
           \end{equation}
     Then, we have that $ \eigenv_n \to\lambda^0_{\bar{n}}$, for some $\bar{n}\in\mathbb{N}$, as  $\varepsilon\to 0$.
    \end{lemma}
     \begin{proof}
     By virtue of estimate \eqref{estieigen}, whose proof will be given in Remark \ref{proofestapriori}, we extract an infinitesimal positive sequence $\{\varepsilon_k\}_{k\in\mathbb{N}}$ such that 
      \begin{equation}
      \label{conveigenva}
      \lambda^{\varepsilon_k}_n\rightarrow \lambda^0_{\bar{n}}, \qquad \varepsilon_k\to 0.
      \end{equation}
     In order to simplify the notation we write $\lambda^{\varepsilon}_n$ in place of $\lambda^{\varepsilon_k}_n$. The normalization condition \eqref{normalcond}, the estimate \eqref{estieigen} and the weak formulation \eqref{weakform} of the spectral problem \eqref{uone}-\eqref{bduzero1} yield
%The goal is to prove  that $\hat{\lambda}^0_n$ is an eigenvalue of the limit problem  \eqref{probv_1^0}-\eqref{bcv_1^01}.

\begin{comment}
how eigenfunctions $\uone, \uzero$ and eigenvalue $\eigenv$ are related to those of limit problem $u^0_1, u^0_0, \lambda^0$. Here we use the spectral problems \eqref{uone}-\eqref{bduzero2} involving eigenfunctions $\uone, \uzero$. We can move from problem involving $\uone$, $\uzero$ to those involving $\vzero$, $\vone$ thanks to changes \eqref{vone}, \eqref{vzero}.\\We require that the eigenfunctions $\uone,\uzero$ satify the following normalization condition
       \begin{equation}
       \label{normcond}
       \|\uone\|^2_{\ballone} + \|\epsilon^{-m}\uzero\|^2_{\ballzero} = 1.
       \end{equation}
\end{comment}

%\begin{align}
%\int_{\ballone} \grad\uone\cdot\grad\phi_1dx-\int_{\boundzero}\partial\uone\phi_1ds &= \eigenv\int_{\ballone} \uone\phi_1dx\label{ct1}\\
%\epsilon^{-1}\int_{\ballzero} \grad\uzero\cdot\grad\phi_0dx+\epsilon^{-1}\int_{\boundzero}\partial\uzero\phi_0ds &= \eigenv\epsilon^{-2m}\int_{\ballzero} \uzero\phi_0dx.\label{ct2}
%\end{align}

      % \begin{equation*}
      % \|\grad\uone\|_{\ballone}^2 + \epsilon^{-1}\|\grad\uzero\|^2_{\ballzero} = \eigenv_n( \|\uone\|^2_{\ballone} + \epsilon^{-2m} \|\uzero\|^2_{\ballzero} ).
      %\end{equation*}
   
       \begin{equation*}
       \label{estgrad}
       \|\grad u^{\varepsilon}_{1,n}\|_{\ballone}^2 + \varepsilon^{-1}\|\grad u^{\varepsilon}_{0,n}\|^2_{\ballzero} = \eigenv_n\leq C_n.
       \end{equation*}
As a consequence,
      \begin{align*}
      \|\grad u^{\varepsilon}_{1,n}\|^2_{\ballone}&\leq C_n, \qquad\|u^{\varepsilon}_{1,n}\|^2_{\ballone}\leq 1.
      \end{align*}
The norms $\|u^{\varepsilon}_{1,n} \|_{H^1(\ballone)}$ are uniformly bounded in $\varepsilon\in (0,\varepsilon_n)$ for a fixed $n$. Then, up to subsequence, $u_{1,n}^{\varepsilon}$ converges weakly in $H^1_0(\ballone, \boundzero)$  and strongly in $L^2(\ballone)$ to some function $g^0_1$, which can be identified as an eigenfunction $u^0_{1,\bar{n}}$ associated to $\lambda^0_{\bar{n}}$. In fact if we take an arbitrary function $\varphi_1\in H^1_0(\ballone, \boundzero)$, and $\varphi_0 = 0$ in $\ballzero$ as a test functions in the integral identity \eqref{weakform}, it admits the limit passage as $\varepsilon\rightarrow 0$, yielding the integral identity
        \begin{equation}
        \label{intidconvthm}
        (\grad g^0_{1}, \grad\varphi_1)_{\ballone} = \lambda^0_{\bar{n}}(g^0_{1}, \varphi_1)_{\ballone}.
        \end{equation}
The equality \eqref{intidconvthm}  gives rise to the problem
         \begin{align*}
        -\laplace g^0_{1}(x) &= \lambda^0_{\bar{n}}g^0_{1}(x), \hspace{0.5cm}x\in\ballone,\\
        \normalderone g^0_{1}(x)=0,\hspace{0.3cm}x\in\boundone, \hspace{0.3cm}&\hspace{0.8cm}
       g^0_{1}(x)=0,\hspace{0.3cm}x\in\boundzero,
        \end{align*}
which implies that $g^0_1= u^0_{1,\bar{n}}$. In other terms,  $\lambda^0_{\bar{n}}$ is an eigenvalue of the limit problem \eqref{probv_1^0}-\eqref{bcv_1^01} with corresponding  eigenfunction $u^0_{1,\bar{n}}$. Concerning the function $u^\varepsilon_{0,n}$, we find that
      \begin{equation*}
      \varepsilon^{-1}\|\grad u^{\varepsilon}_{0,n}\|^2_{\ballzero} \leq C, \qquad
      \varepsilon^{-2m}\|u^{\varepsilon}_{0,n}\|^2_{\ballzero} \leq 1,
      \end{equation*}
so that $u^{\varepsilon}_{0,n}$ converges to $0$  strongly in $H^1_0(\ballzero)$ and hence in $L^2(\ballzero)$ (if necessary, we can again pass to a subsequence). \\ The eigenfunction $u^0_{1,\bar{n}}$ is also normalized in $L^2$-norm. Indeed, bearing in mind the replacement \eqref{vzero}, we deduce that
    \begin{equation*}
    \|v^{\varepsilon}_{0,n} \|_{\ballzero}^2\leq 1, \qquad \|\grad v^{\varepsilon}_{0,n} \|_{\ballzero}^2\leq \varepsilon^{1-2m}C,
    \end{equation*}
from which it follows that $v^{\varepsilon}_{0,n}$ converges strongly in $L^2(\ballzero)$ to some constant $\tilde{c}$. In order to prove that $\tilde{c}=0$, we take  $\varphi_1=\varphi_0=\varepsilon^m\tilde{c}$ as test functions in \eqref{weakform}, obtaining
    \begin{equation*}
    0= \eigenv_n \left(\varepsilon^m\tilde{c}\int_{\ballone} u^{\varepsilon}_{1,n}dx + \tilde{c}\int_{\ballzero} v^{\varepsilon}_{0,n}dx\right).
    \end{equation*}
Passing to the limit as $\varepsilon\to 0$, we find that $\tilde{c}=0$. As a consequence, $\varepsilon^{-2m}\|u^{\varepsilon}_{0,n}\|_{\ballzero}^2\to 0,$ as $\varepsilon\to 0$ and the normalization condition \eqref{normalcond} leads to 
    $
    \|u_{1,\bar{n}}^0\|_{L^2(\ballone)}=1.
    $
  \end{proof}
 %This implies that there exist an index $p$ such that $\hat{\lambda}^0_n = \lambda^0_p$ and $\hat{u}^0_{1,n} = u^0_{1,p}$

The goal of the next subsection is to check that $n=\bar{n}$, concluding hence the proofs of Proposition \ref{prop} and of Theorem \ref{thmmain}.

\subsubsection{Step 2: Lemma about near eigenvalues and eigenfunctions}
Let $\newspace$ denote the Hilbert space $H^1(\Omega)$ endowed with the inner product 
      \begin{equation}
      \label{newinnerprod}
      \langle U,V\rangle_{\varepsilon} = (\grad U_1,\grad V_1)_{\ballone} + \varepsilon^{-1}(\grad U_0,\grad V_0)_{\ballzero} + (U_1, V_1)_{\ballone} +\varepsilon^{-2m}( U_0,V_0)_{\ballzero}.
      \end{equation}
We introduce the operator $\visikop$ in $\newspace$ by the formula
       \begin{equation}
       \label{defnewop}
       \langle \visikop U, V\rangle_{\varepsilon} = (U_1,V_1)_{\ballone} + \varepsilon^{-2m}(U_0, V_0)_{\ballzero} \qquad \forall\hspace{0.08cm} U,V\in\newspace,
       \end{equation}
and the new spectral parameter 
        \begin{equation}
        \label{newspecpar}
        k^{\varepsilon}_n = (1+\eigenv_n)^{-1}.
        \end{equation}
It is easy to verify that $\visikop$ is a continuous, self-adjoint, positive and compact operator. Thus, the spectrum of operator $\visikop$ consists of the essential spectrum $\sigma_{\text{ess}}(\visikop) = \{0\}$ %(cf. \cite[ Theorems 10.1.5 and 10.2.2]{BS87}) 
and an infinitesimal positive sequence of real eigenvalues 
    \begin{equation*}
    k^{\varepsilon}_1\geq k^{\varepsilon}_2\geq \cdots \geq k^{\varepsilon}_n\geq\cdots \rightarrow 0.
    \end{equation*}
Taking into account formulas \eqref{newinnerprod}-\eqref{newspecpar}, the integral identity \eqref{weakform} is equivalent to the abstract equation
      % \begin{equation*}
       %\langle U^{\epsilon}, \varPhi \rangle_{\epsilon} - \langle \visikop U^{\epsilon}, \varPhi\rangle _{\epsilon} = \eigenv \langle \visikop U^{\epsilon}, \varPhi\rangle _{\epsilon}, \qquad \varPhi \in \newspace,
      %\end{equation*}
%i.e.
      \begin{equation*}
       \visikop U^{\varepsilon} = k^{\varepsilon}_n U^{\varepsilon}.
      \end{equation*}
The following statement is known as \textquotedblleft lemma  about near eigenvalues and eigenvectors\textquotedblright (cf. \cite{VL57}) and follows from the spectral decomposition of the resolvent, cf. \cite[Chapter 6]{BS87}.
   \begin{lemma}
   \label{Visiklemma}
   Assume $\almosteigenfun\in\newspace$ and $\approxeigenv\in\mathbb{R}_{+}$ such that
          \begin{equation*}
            \|\almosteigenfun\|_{\newspace}=1, \qquad \|\visikop\almosteigenfun - \approxeigenv\almosteigenfun\|_{\newspace} =: \delta^{\varepsilon}\in (0, \approxeigenv).
          \end{equation*}
   Then in the segment $[\approxeigenv-\delta^{\varepsilon}, \approxeigenv+\delta^{\varepsilon}]$ there is at least one eigenvalue of the operator $\visikop$. Moreover, for any $\delta^\prime_{\varepsilon}\in(\delta^{\varepsilon}, \approxeigenv)$ there exist coefficients $\visikvec_{J^{\varepsilon}}, \cdots, \visikvec_{J^{\varepsilon}+K^{\varepsilon}-1}$ such that 
          \begin{equation*}
          \label{secondpartVisiklemma}
          \biggl\|\almosteigenfun - \sum_{j=J^{\varepsilon}}^{J^{\varepsilon}+K^{\varepsilon}-1}  \visikvec_j\funcnewop_j\biggr\|_{\newspace}\leq 2\frac{\delta^{\varepsilon}}{\delta^\prime_{\varepsilon}}, \qquad  \sum_{j=J^{\varepsilon}}^{J^{\varepsilon}+K^{\varepsilon}-1}  |\visikvec_j|^2=1,
          \end{equation*}
   where $\funcnewop_{J^{\varepsilon}}, \cdots, \funcnewop_{J^{\varepsilon}+K^{\varepsilon}-1}$ are eigenvectors associated to all eigenvalues $k^{\varepsilon}_{J^{\varepsilon}}, \cdots, k^{\varepsilon}_{J^{\varepsilon}+K^{\varepsilon}-1}$ of the operator $\visikop$ situated in $[\approxeigenv-\delta^{\varepsilon}, \approxeigenv+\delta^{\varepsilon}]$. The eigenvectors are subject to the orthonormalization conditions
           \begin{equation}
           \label{orthnor}
           \langle \funcnewop_i, \funcnewop_j\rangle_{\varepsilon}=\delta_{i,j}.
           \end{equation}
   \end{lemma} 
%We will construct an approximate eigenvalues $\approxeigenv$ and eigenfunctions $\almosteigenfun$. 
 %The pair $\almosteigenval$ and $\almosteigenfun$ are chosen as approximate solutions of the spectral problem \eqref{defnewop}. 
\noindent In the case of a simple eigenvalue $\lambda^0_{\bar{n}}$ of the problem \eqref{probv_1^0}-\eqref{bcv_1^01}, the approximate eigenvalue $\almosteigenval_{\bar{n}}$ is
%The approximate eigenvalue  is suggested by asymptotic expansions \eqref{asyeigen}: for every $\lambda^0_n$ we set
     \begin{equation}
     \label{approeigenval}
     (1+\lambda^0_{\bar{n}} + \varepsilon^{2m}\lambda^\prime_{\bar{n}})^{-1},
     \end{equation}
where $\lambda^\prime_{\bar{n}}$ is the asymptotic correction \eqref{singleeigenv} and %in the case of simple eigenvalue $\lambda^0_n$ or \eqref{multipleeigenv} in the case of multiple eigenvalue $\lambda^0_n$ . If $\lambda^0_n$ is a simple eigenvalue, 
the approximate eigenfunction $\almosteigenfun_{\bar{n}} = (\mathfrak{U}^\varepsilon_{0,\bar{n}}, \mathfrak{U}^\varepsilon_{1,\bar{n}})$ is defined by
%By formulas \eqref{asyu_0}, \eqref{asyu_1}, we choose as the approximate eigenfunction
     \begin{equation}
     \label{approeigenfnc}
      (\varepsilon^{2m}c_{0,\bar{n}}+ \varepsilon u^\prime_{0,\bar{n}}+\varepsilon^{2-2m} \fu^\varepsilon_{0,\bar{n}}, \hspace{0.2cm} u^0_{1,\bar{n}} +\varepsilon^{2m} u^\prime_{1,\bar{n}} + \varepsilon \fu_{1,\bar{n}}+\varepsilon^{2-2m}\su_{1,\bar{n}}),
     \end{equation}
where $c_{0,\bar{n}}$ is given by \eqref{constant}, $u^\prime_{0,\bar{n}}$ is the solution to the problem \eqref{pbv'_0}, $u^0_{1,\bar{n}}$ solves the limit problem \eqref{probv_1^0}-\eqref{bcv_1^01} and $u^\prime_{1,\bar{n}}$ is characterized by the problem \eqref{pbv'_1}-\eqref{bcv'}. The arbitrary (but fixed) functions $\fu_{1,\bar{n}}, \su_{1,\bar{n}}$ in $H^1(\ballone)$ are such that 
    \begin{equation*}
    \fu_{1,\bar{n}}(x) = u^\prime_{0,\bar{n}}(x), \qquad \su_{1,\bar{n}}(x) = \fu^\varepsilon_{0,\bar{n}} (x),\qquad x\in\boundzero,
    \end{equation*}
and $\fu^\varepsilon_{0,\bar{n}}$ is the solution to the Neumann problem for the Helmholtz operator 
    \begin{align}
    -\laplace\fu^\varepsilon_{0,\bar{n}} (x) -\varepsilon^{1-2m}\lambda^0_{\bar{n}}\fu^\varepsilon_{0,\bar{n}}(x)&= \lambda^0_{\bar{n}}u'_{0,\bar{n}}(x), \qquad x\in \ballzero, \label{prou''}\\
    \normalderzero \fu^\varepsilon_{0,\bar{n}}(x)&=0, \hspace{1.9cm} x\in\boundzero. \label{bcu''}
    \end{align} 
Denoting by $L^2_{\bot}(\ballzero)$ the subspace $\{u\in L^2(\ballzero)\hspace{0.02cm}:\hspace{0.02cm} \int_{\ballzero} u(x)dx=0  \}$ and setting $H^2_{\bot}(\ballzero) = H^2(\ballzero)\cap L^2_{\bot}(\ballzero)$, the Neumann Laplacian 
$\laplace : H^2_{\perp}(\ballzero)\rightarrow L^2_{\perp}(\ballzero) $ is an isomorphism. Consequently, for small $\varepsilon>0$ the mapping $-\laplace-\varepsilon^{1-2m}\lambda^0_{\bar{n}}\text{Id}$ is also an isomorphism, i.e. $\fu_{0,\bar{n}}^\varepsilon$ is the unique solution to the problem \eqref{prou''}-\eqref{bcu''} (cf., e.g., \cite[Theorem 3.6.1]{HPC05}). Furthermore, the estimate 
     \begin{equation*}
     \|\fu_{0,\bar{n}}^\varepsilon\|_{H^2_{\perp}(\ballzero)}\leq c\lambda^0_n\|u'_{0,\bar{n}}\|_{L^2_{\perp}(\ballzero)}
     \end{equation*}
holds, where the constant $c$ is independent of the parameter $\varepsilon$. 

% % % % % % % %AGGIUNGERE ALLA TESI % % % % % % % % % 
\begin{comment}
The inverse linear map $(-\laplace-\varepsilon^{1-2m}\lambda^0_n\text{Id})^{-1}$ is given by the Neumann series (e.g. \cite[Theorem 3.6.1]{HPC05})
     \begin{equation*}
     (-\laplace-\varepsilon^{1-2m}\lambda^0_n\text{Id})^{-1} = [\text{Id}-\varepsilon^{1-2m}\lambda^0_n(-\laplace)^{-1}]^{-1}(-\laplace)^{-1} = \sum_{k=0}^{\infty} (\varepsilon^{1-2m}\lambda^0_n)^k K^{k+1},
     \end{equation*}
with $K := (-\laplace)^{-1} $. Therefore, the solution $\fu^\varepsilon_{0,n}$ to the problem \eqref{prou''}-\eqref{bcu''} admits the expansion
     \begin{equation}
     \label{expalmeig}
     \fu^\varepsilon_{0,n}(x) = \sum_{k=0}^{\infty} (\varepsilon^{1-2m}\lambda^0_n)^k K^{k+1} (\lambda^0_nu'_{0,n}(x)) =\sum_{k=0}^{\infty} (\varepsilon^{1-2m})^k u^{(k)}_{0,n}(x).  
     \end{equation}
Here $ u^{(k)}_{0,n}(x) = (\lambda^0_n)^k K^{k+1} (\lambda^0_nu'_{0,n}(x))$ and it is the solution to the integral equality
     \begin{equation*}
     (\grad u^{(k)}_{0,n}, \grad w_0)_{\ballzero} = \lambda^0_n ( u^{(k-1)}_{0,n}, w_0)_{\ballzero}, \qquad \forall w_0\in H^1(\ballzero), \quad k=0,1,2,\cdots,
     \end{equation*}
where we set $u^{(-1)}_{0,n}(x) := u'_{0,n}(x)$.\\
%Besides, we require that $u''_{0,n}$ has zero mean
 %    \begin{equation*}
  %   \overline{u}''_{0,n}=\frac{1}{|\ballzero|}\int_{\ballzero}u''_{0,n}(x)dx = 0.
   %  \end{equation*}
\end{comment}
% % % % % % % % % % % % % % % % % % % % % % % % % % %

\noindent If $\lambda^0_{\bar{n}}$ is a multiple eigenvalue (cf. \eqref{multipleig}) and $\lambda^\prime_{\bar{n}}$ in \eqref{approeigenval} is given by \eqref{multipleeigenv}, then the functions $u^0_{1,j}, c_{0,j}, u^\prime_{1,j}$ in  \eqref{approeigenfnc} are replaced with $V^0_{1,j}, v^0_{0,j}$ defined by the formulas \eqref{lincom}, \eqref{c_{0,k}} and the solution $V^\prime_{1,j}$ to the problem \eqref{multiple}-\eqref{multiple2} for any $j=\bar{n},\cdots,\bar{n}+\tau-1$. \\The almost eigenfunction $\almosteigenfun_{\bar{n}}$ belongs to Hilbert space $\newspace$ but in generally it does not satisfy the normalization condition. Then, we apply Lemma \ref{Visiklemma} with $ \|\almosteigenfun_{\bar{n}}\|^{-1}_{\newspace}\almosteigenfun_{\bar{n}}\in\newspace$. Note that for sufficiently small $\varepsilon$ the estimate
      \begin{equation}
      \label{normvisik}
      \|\almosteigenfun_{\bar{n}}\|_{\newspace} \geq\frac{1}{2},
      \end{equation}
follows from formula \eqref{estnormalmeig}. Indeed, %introducting the notations
    %   \begin{equation*}
     %  \mathfrak{U}^\varepsilon_{1,i} = u^0_{1,i} +\varepsilon^{2m}u^\prime_{1,i} +\varepsilon \fu_{1,i}+\varepsilon^{2-2m}\su_{1,i}, \qquad \mathfrak{U}^\varepsilon_{0,i} = \varepsilon^{2m}c_{0,i} +\varepsilon u^\prime_{0,i} + \varepsilon^{2-2m}\fu^\varepsilon_{0,i},
      % \end{equation*}
the inner product 
      \begin{align}
      \langle \almosteigenfun_i, \almosteigenfun_j\rangle_{\varepsilon} & = (\grad \mathfrak{U}^\varepsilon_{1,i},\hspace{0.2cm} \grad\mathfrak{U}^\varepsilon_{1,j}  )_{\ballone} +\varepsilon^{-1}(\grad\mathfrak{U}^\varepsilon_{0,i},\hspace{0.2cm}  \grad\mathfrak{U}^\varepsilon_{0,j})_{\ballzero}+ (\mathfrak{U}^\varepsilon_{1,i},\hspace{0.2cm} \mathfrak{U}^\varepsilon_{1,j})_{\ballone}\notag\\
      &\quad + \varepsilon^{-2m} (\mathfrak{U}^\varepsilon_{0,i},\hspace{0.2cm} \mathfrak{U}^\varepsilon_{0,j})_{\ballzero}\notag\\
      %&= (\grad u^0_{1,i}, \grad u^0_{1,j})_{\ballone} +(u^0_{1,i}, u^0_{1,j})_{\ballone} + \varepsilon^{2m} I^{i,j}_1+ \varepsilon^{4m} [ (\grad u'_{1,i},\grad u'_{1,j})_{\ballone}+ ( u'_{1,i}, u'_{1,j})_{\ballone}   ]\notag\\
     % &\quad+ \varepsilon I^{i,j}_2+ \varepsilon^{2m+1}I^{i,j}_3 + \varepsilon^{2-2m} I^{i,j}_4 + \varepsilon^2 I^{i,j}_5+\varepsilon^{3-4m}I^{i,j}_6+ \varepsilon^{4-6m} (\fu^\varepsilon_{0,i}, \fu^\varepsilon_{0,j})_{\ballzero}\notag\\ 
    %  &\quad+\varepsilon^{3-2m}I^{i,j}_7 +\varepsilon^{4-4m}[ (\grad\su_{1,i}, \grad\su_{1,j})_{\ballone} +  (\su_{1,i}, \su_{1,j})_{\ballone}    ] \notag\\    
      &= (\grad u^0_{1,i}, \grad u^0_{1,j})_{\ballone} + (u^0_{1,i}, u^0_{1,j})_{\ballone} + O(\varepsilon^{2m})\notag\\
     &= (1+\lambda^0_p)(u^0_{1,p}, u^0_{1,q})_{\ballone} + O(\varepsilon^{2m})\notag\\
      &= (1+\lambda^0_i)\delta_{ij} + O(\varepsilon^{2m}),\hspace{2cm} i,j=1,2,\dots\label{estnormalmeig}
      \end{align}

\noindent where the last equality is due to orthonormalization conditions \eqref{orthnormv^0_1}. Note that $O(\varepsilon^{2m})$ contains the terms listed below multiplied by some power of $\varepsilon$ and they can be easily estimated:
      \begin{align*}
      \varepsilon^{2m} &: (\grad u^0_{1,i}, \grad u^0_{1,j})_{\ballone}  +(\grad u^\prime_{1,i}, \grad u^0_{1,j})_{\ballone} + (u^0_{1,i},  u^\prime_{1,j})_{\ballone} + ( u^\prime_{1,i},  u^0_{1,j})_{\ballone}; \\
      \varepsilon^{4m} &: (\grad u'_{1,i},\grad u'_{1,j})_{\ballone}+ ( u'_{1,i}, u'_{1,j})_{\ballone}; \qquad \varepsilon^{4-6m}:(\fu^\varepsilon_{0,i}, \fu^\varepsilon_{0,j})_{\ballzero};\\ 
      \varepsilon \hspace{0.4cm}&:(\grad u^0_{1,i}, \grad\fu_{1,j})_{\ballone} +(\grad \su_{1,i}, \grad u^0_{1,j})_{\ballone}+ (\grad u^\prime_{0,i}, \grad u^\prime_{0,j})_{\ballzero} \\
    &\quad  + ( u^0_{1,i}, \fu_{1,j})_{\ballone}+ (\fu_{1,i}, u^0_{1,j})_{\ballone};\\
      \varepsilon^{2m+1} &: (\grad u^\prime_{1,i} , \grad\fu_{1,j})_{\ballone} + (\grad \fu_{1,i} , \grad\fu_{1,j})_{\ballone} + (u^\prime_{1,i} , \fu_{1,j})_{\ballone} + (\su_{1,i} , u^\prime_{1,j})_{\ballone}; \\
     \varepsilon^{2-2m} &: (\grad u^0_{1,i}, \grad\su_{1,j})_{\ballone}  + (\grad \su_{1,i}, \grad u^0_{1,j})_{\ballone} + (\grad u^\prime_{0,i}, \grad \fu^\varepsilon_{0,j})_{\ballzero}  \\
      &\quad  +  (\grad\fu^\varepsilon_{0,i}, \grad u^\prime_{0,j})_{\ballzero}+( u^0_{1,i}, \su_{1,j})_{\ballone} + ( \su_{1,i}, u^0_{1,j})_{\ballone} + ( u^\prime_{0,i}, u^\prime_{0,j})_{\ballzero};\\
      \varepsilon^{2} &:  (\grad u^\prime_{1,i} , \grad\su_{1,j})_{\ballone}  + (\grad \fu_{1,i} , \grad\fu_{1,j})_{\ballone}   +(\grad \su_{1,i} , \grad u^\prime_{1,j})_{\ballone} \\
      &\quad+(u^\prime_{1,i} , \su_{1,j})_{\ballone}+(\fu_{1,i} , \fu_{1,j})_{\ballone} + (\su_{1,i} , \su_{1,j})_{\ballone}; \\
       \varepsilon^{3-4m} &: (\grad \fu^\varepsilon_{0,i}, \grad \fu^\varepsilon_{0,j})_{\ballzero}  +(u^\prime_{0,i},  \fu^\varepsilon_{0,j})_{\ballzero} + (\fu^\varepsilon_{0,i},  u^\prime_{0,j})_{\ballzero};\\
      \varepsilon^{4-4m} &: (\grad \fu_{1,i}, \grad\su_{1,j} )_{\ballone} +(\grad \su_{1,i}, \grad\fu_{1,j} )_{\ballone}+( \fu_{1,i}, \su_{1,j} )_{\ballone} + ( \su_{1,i}, \fu_{1,j} )_{\ballone};\\
      \varepsilon^{4-4m} &:(\grad\su_{1,i}, \grad\su_{1,j})_{\ballone} +  (\su_{1,i}, \su_{1,j})_{\ballone}.  
     \end{align*}
Consequently, we obtain

\noindent 
      \begin{align*}
      \delta^{\varepsilon}_{\bar{n}} &=  \|\almosteigenfun_{\bar{n}}\|^{-1}_{\newspace}\|\visikop \almosteigenfun_{\bar{n}} -\approxeigenv_{\bar{n}}\almosteigenfun_{\bar{n}}\|_{\newspace}\\
      &=\|\almosteigenfun_{\bar{n}}\|^{-1}_{\newspace}\sup_{\substack{W_{\varepsilon}\in\newspace\\\|W_{\varepsilon}\|_{\newspace}=1}} |\langle\visikop \almosteigenfun_{\bar{n}} -\approxeigenv_{\bar{n}}\almosteigenfun_{\bar{n}}, W^{\varepsilon}\rangle_{\varepsilon}|\\
      &=\|\almosteigenfun_{\bar{n}}\|^{-1}_{\newspace}(\almosteigenval_{\bar{n}})^{-1}\sup_{\substack{W_{\varepsilon}\in\newspace\\\|W_{\varepsilon}\|_{\newspace}=1}} |(\almosteigenval_{\bar{n}})^{-1}\langle\visikop \almosteigenfun_{\bar{n}}, W^{\varepsilon}\rangle_{\varepsilon}- \langle\almosteigenfun_{\bar{n}}, W^{\varepsilon}\rangle_{\varepsilon}|\\
      &\leq c\sup_{\substack{W_{\varepsilon}\in\newspace\\\|W_{\varepsilon}\|_{\newspace}=1}} |(\almosteigenval_{\bar{n}})^{-1}\langle\visikop \almosteigenfun_{\bar{n}}, W^{\varepsilon}\rangle_{\varepsilon}- \langle\almosteigenfun_{\bar{n}}, W^{\varepsilon}\rangle_{\varepsilon}|,
      \end{align*}
where in the last inequality we used \eqref{normvisik} and $(\almosteigenval_{\bar{n}})^{-1}\geq 1$. Now, we focus only on the absolute value. Using formulas \eqref{newinnerprod} and \eqref{defnewop}, we find
    \begin{align}
    |(\almosteigenval_{\bar{n}})^{-1}\langle\visikop \almosteigenfun_{\bar{n}} ,W^{\varepsilon}\rangle_{\varepsilon}- \langle\almosteigenfun_{\bar{n}}, W^{\varepsilon}\rangle_{\varepsilon}| & %= |(\lambda^0_n+\varepsilon^{2m}\lambda^\prime_n)[(\mathfrak{U}^\varepsilon_{1,n}, W^{\varepsilon}_1)_{\ballone} + \varepsilon^{-2m}(\mathfrak{U}^\varepsilon_{0,n}, W^{\varepsilon}_0)_{\ballzero}]\notag\\
    %&\quad-(\grad\mathfrak{U}^\varepsilon_{1,n}, \grad W^{\varepsilon}_1)_{\ballone}- \varepsilon^{-1}(\grad\mathfrak{U}^\varepsilon_{0,n}, W^{\varepsilon}_0)_{\ballzero}|\notag\\
    =| J_0 +\varepsilon^{2m} J_1 + \varepsilon^{1-2m}J_2+ \varepsilon^{4m}\lambda^\prime_{\bar{n}}(u^\prime_{1,\bar{n}},W^{\varepsilon}_1)_{\ballone} \notag\\
    &\quad + \varepsilon^{2-4m}\lambda^0_{\bar{n}}(\fu^\varepsilon_{0,\bar{n}} ,W^{\varepsilon}_0)_{\ballzero} +\varepsilon J_3 +\varepsilon^{1+2m} \lambda^\prime_{\bar{n}}  (\fu_{1,\bar{n}},W^{\varepsilon}_1)_{\ballone} \notag\\
    &\quad \varepsilon^{2-2m}J_4 + \varepsilon^2\lambda^\prime_{\bar{n}} (\su_{1,\bar{n}},W^{\varepsilon}_1)_{\ballone}|.\label{firsthp}
    \end{align}
Here 
    \begin{align*}
    J_0 &=\lambda^0_{\bar{n}}(u^0_{1,\bar{n}},W^{\varepsilon}_1)_{\ballone}+\lambda^0_{\bar{n}}(c_{0,\bar{n}} ,W^{\varepsilon}_0)_{\ballzero} -(\grad u^0_{1,\bar{n}}, \grad W^{\varepsilon}_1)_{\ballone} - (\grad u^\prime_{0,\bar{n}}, \grad W^{\varepsilon}_0)_{\ballzero};\\
    J_1 &=\lambda^\prime_{\bar{n}}(c_{0,\bar{n}}, W^{\varepsilon}_0)_{\ballzero}+ \lambda^0_{\bar{n}}(u^\prime_{1,\bar{n}} ,W^{\varepsilon}_1)_{\ballone} + \lambda^\prime_{\bar{n}}(u^0_{1,{\bar{n}}},W^{\varepsilon}_1)_{\ballone}- (\grad u^\prime_{1,\bar{n}}, \grad W^{\varepsilon}_1)_{\ballone};\\
    J_2 &= \lambda^0_{\bar{n}}(u^\prime_{0,\bar{n}} ,W^{\varepsilon}_0)_{\ballzero}-(\grad\fu^\varepsilon_{0,\bar{n}},\grad W^{\varepsilon}_0)_{\ballzero};\\
    J_3&=\lambda^0_{\bar{n}}(\fu_{1,\bar{n}}, W^{\varepsilon}_1)_{\ballone} +\lambda^\prime_{\bar{n}}(u'_{0,\bar{n}}, W^{\varepsilon}_0)_{\ballzero}-(\grad \fu_{1,\bar{n}}, \grad W^{\varepsilon}_1)_{\ballone};\\
    J_4&=  \lambda^0_{\bar{n}}(\su_{1,\bar{n}}, W^{\varepsilon}_1)_{\ballone} +\lambda^\prime_{\bar{n}}(\fu^\varepsilon_{0,\bar{n}}, W^{\varepsilon}_0)_{\ballzero}-(\grad \su_{1,\bar{n}}, \grad W^{\varepsilon}_1)_{\ballone}.
    \end{align*}

\noindent Integrating by parts the problems \eqref{probv_1^0}-\eqref{bcv_1^01}, \eqref{pbv'_0} and \eqref{pbv'_1}-\eqref{bcv'},  the expression under the modulus sign on the right-hand side of \eqref{firsthp} becomes
% % % % % % % AGGIUNGERE ALLA TESI % % % % % % % %
 %    \begin{align*}
 %    (\grad u^0_{1,n}, \grad W^{\varepsilon}_1)_{\ballone}&= \lambda^0_n(u^0_{1,n},W^{\varepsilon}_1)_{\ballone}+ (\normalder u^0_{1,n}, W^{\varepsilon}_1)_{\boundzero}, \\%\label{formula1}\\
 %    (\grad u^\prime_{0,n}, \grad W^{\varepsilon}_0)_{\ballzero} &= \lambda^0_n(c_{0,n}, W^{\varepsilon}_0)_{\ballzero}- (\normalder u^0_{1,n}, W^{\varepsilon}_0)_{\boundzero} , \\ %\label{formula2}\\
 %    (\grad u^\prime_{1,n}, \grad W^{\varepsilon}_1)_{\ballone}&= \lambda^\prime_n(u^0_{1,n}, W^{\varepsilon}_1)_{\ballone} + \lambda^0_n(u^\prime_{1,n}, W^{\varepsilon}_1)_{\ballone} +(\normalder u^\prime_{1,n}, W^{\varepsilon}_1)_{\boundzero} . %\label{f3}
 %    \end{align*}
%in \eqref{firsthp}, we find 
          \begin{align}
          |(\almosteigenval_{\bar{n}})^{-1}\langle\visikop \almosteigenfun_{\bar{n}} ,W^{\varepsilon}\rangle_{\varepsilon}- \langle\almosteigenfun_{\bar{n}}, W^{\varepsilon}\rangle_{\varepsilon}|& = |\varepsilon^{2m}J'_1+\varepsilon^{1-2m}J_2+\varepsilon^{4m}\lambda^\prime_{\bar{n}}(u^\prime_{1,\bar{n}}, W^{\varepsilon}_1)_{\ballone} \notag\\
          &\quad+\varepsilon^{2-4m}\lambda^0_{\bar{n}}(\fu^\varepsilon_{0,\bar{n}}, W^{\varepsilon}_0)_{\ballzero}+\varepsilon J_3+ \varepsilon^{1+2m}\lambda^\prime_{\bar{n}}(\fu_{1,\bar{n}}, W^{\varepsilon}_1)_{\ballone}\notag\\
          &\quad +\varepsilon^{2-2m}J_4 + \varepsilon^2\lambda^\prime_{\bar{n}} (\su_{1,\bar{n}},W^{\varepsilon}_1)_{\ballone}|,\notag
          \end{align}
with 
      $
      J'_1 = \lambda^\prime_{\bar{n}}(c_{0,\bar{n}}, W^{\varepsilon}_0)_{\ballzero}-(\normalderzero u^\prime_{1,\bar{n}}, W^{\varepsilon}_1)_{\boundzero}.
      $
Note that $\varepsilon^{1-2m}J_2+\varepsilon^{2-4m}\lambda^0_{\bar{n}}(\fu^\varepsilon_{0,\bar{n}}, W^{\varepsilon}_0)_{\ballzero}=0$ due to the fact that $\fu^\varepsilon_{0,\bar{n}}$ can be written as Neumann series (cf. \cite[Theorem 3.6.1]{HPC05}).

\noindent Moreover, the definition of the inner product \eqref{newinnerprod} in the Hilbert space $\newspace$ yields the following estimates of the classical norm in $L^2(\Omega_i),$ $i=0,1$ 
      \begin{equation*}
      \|W^{\varepsilon}_1\|_{\ballone}\leq\|W\|_{\newspace}, \quad \|\grad W^{\varepsilon}_1\|_{\ballone}\leq \|W^{\varepsilon}\|_{\newspace}, \quad \|W^{\varepsilon}\|_{\ballzero}\leq \varepsilon^{m}\|W^{\varepsilon}\|_{\newspace}.
      \end{equation*}
Finally,
     \begin{align*}
     \delta^{\varepsilon}_{\bar{n}} %& \leq c\sup_{\substack{W_{\varepsilon}\in\newspace\\\|W_{\varepsilon}\|_{\newspace}=1}} |(\almosteigenval)^{-1}\langle\visikop \almosteigenfun, W^{\varepsilon}\rangle_{\varepsilon}- \langle\almosteigenfun, W^{\varepsilon}\rangle_{\varepsilon}|\\
     &\leq C_1\varepsilon^{3m} + C_2\varepsilon +C_3\varepsilon^{4m}+ C_4\varepsilon^{1+m}+C_5\varepsilon^{1+2m}+ C_6\varepsilon^{2-2m} +C_7\varepsilon^{2-m}+ C_8\varepsilon^2\leq C\varepsilon^{\gamma},
     \end{align*} 
where $\gamma= \min\{3m, 1\}$. Then, the first part of Lemma \ref{Visiklemma} implies that there exists at least one eigenvalue $k^{\varepsilon}_{n}$ of  $\visikop$ such that 
    \begin{equation}
    \label{f2}
    |\almosteigenval_{\bar{n}} - k^{\varepsilon}_{n}|\leq C\varepsilon^{\gamma}.
    \end{equation}
Bearing in mind Lemma \ref{lemma}, the inequality \eqref{f2} can be written as
     \begin{equation}
     \label{justest}
     |\lambda^{\varepsilon}_{n} - \lambda^0_{\bar{n}} -\varepsilon^{2m}\lambda^\prime_{\bar{n}}| \leq C |1+\eigenv_{n}||1+\lambda^0_{\bar{n}}+\varepsilon^{2m}\lambda^\prime_{\bar{n}}| \varepsilon^{\gamma}\leq C_n\varepsilon^\gamma \qquad\forall\varepsilon\in (0, \varepsilon_N).
     \end{equation}
In order to show \eqref{justasy} and to conclude the proof of Theorem \ref{thmmain}, we must check that the indices $n$ and $\bar{n}$ in the inequality \eqref{justest} coincide. To this end, we will apply the second part of Lemma \ref{Visiklemma}.\\Assume that $\lambda^0_{\bar{n}}$ is an eigenvalue of multiplicity $\tau\geq 2$ of the problem \eqref{probv_1^0}-\eqref{bcv_1^01}. We can associate $\tau$ copies of almost eigenfunctions $\almosteigenfun_p$, given by \eqref{approeigenfnc}, and introduce
     \begin{equation*}
     \label{seconddelta}
     \delta^\prime_{\bar{n}\varepsilon} = T\max\{\delta_{\bar{n}}^{\varepsilon}, \cdots,\delta_{\bar{n}+\tau-1}^{\varepsilon}\},
     \end{equation*}
where $T$ is large and fixed (independent of $\varepsilon$). The second part of Lemma \ref{Visiklemma} gives the normalized columns $\column_{J^{\varepsilon}}, \cdots, \column_{J^{\varepsilon}+K^{\varepsilon}-1}$ verifying the inequality 
      \begin{equation}
      \label{inalmoeigenfun}
      \biggl\|\almosteigenfun_p - \sum_{j=J^{\varepsilon}}^{J^{\varepsilon}+K^{\varepsilon}-1}  \column_j\funcnewop_j\biggr\|_{\newspace}\leq \frac{2}{T}, \qquad p=\bar{n}, \cdots, \bar{n}+\tau-1.
      \end{equation}
%and the columns $\column$ satisfy
     % \begin{equation}
     % \label{vl1}
     % \column = (\column_{J^{\varepsilon}}, \cdots, \column_{J^{\varepsilon} + K^{\varepsilon}-1})\in\mathbb{C}^{K^{\varepsilon}}, \qquad |\column|=1, \qquad p=n, \cdots, n+\tau-1.
     % \end{equation}
We aim to show that in the closure of a $\delta^\prime_{\bar{n}\varepsilon}$-neighbourhood of the point $\almosteigenval_{\bar{n}}$ there are at least $\tau$ eigenvalues  of the operator $\visikop$, i.e. in the inequality \eqref{inalmoeigenfun} the number $K^{\varepsilon}$ is such that $K^{\varepsilon}\geq\tau$. 
The equality \eqref{estnormalmeig} implies the estimate
    \begin{equation}
    \label{stimaeigenfun}
    |\langle \almosteigenfun_p, \almosteigenfun_q\rangle_{\varepsilon} - (1+\lambda^0_p)\delta_{p,q}|\leq C\varepsilon^{2m}, \qquad p,q=\bar{n},\dots,\bar{n}+\tau-1.
    \end{equation}
Set 
    \begin{equation*}
    S^{\varepsilon}_p =\sum_{j=N^{\varepsilon}}^{N^{\varepsilon} + K^{\varepsilon}-1}\column_j\funcnewop_j,\qquad p=\bar{n},\dots,\bar{n}+\tau-1.
    \end{equation*}
In view of the estimate \eqref{stimaeigenfun} and orthonormalization condition \eqref{orthnor} of $\funcnewop$, we find
     \begin{align*}
     \biggl|\sum_{j=N^{\varepsilon}}^{N^{\varepsilon}+K^{\varepsilon}-1} \column_j\overline{\columnq_j}- (1+\lambda^0_{\bar{n}})\delta_{p,q}  \biggr| &= \biggl| \langle  \sum_{j=N^{\varepsilon}}^{N^{\varepsilon}+K^{\varepsilon}-1} a^{p\varepsilon}_ju^{\varepsilon}_j, \sum_{j=N^{\varepsilon}}^{N^{\varepsilon}+K^{\varepsilon}-1} a^{q\varepsilon}_ju^{\varepsilon}_j\rangle_{\varepsilon} - (1+\lambda^0_{\bar{n}})\delta_{p,q}  \biggr| \\
     &= |\langle S^{\varepsilon}_p,S^{\varepsilon}_q\rangle_{\varepsilon} - \langle S^{\varepsilon}_p, \almosteigenfun_q\rangle_{\varepsilon} + \langle S^{\varepsilon}_p, \almosteigenfun_q\rangle_{\varepsilon} - \langle \almosteigenfun_p, \almosteigenfun_q\rangle_{\varepsilon} \\
     &\quad  +\langle \almosteigenfun_p, \almosteigenfun_q\rangle_{\varepsilon} -(1+\lambda^0_{\bar{n}})\delta_{p,q}  |\\
     &\leq |\langle S^{\varepsilon}_p,S^{\varepsilon}_q-\almosteigenfun_q\rangle_{\varepsilon}  + \langle S^{\varepsilon}_p-\almosteigenfun_p, \almosteigenfun_q\rangle_{\varepsilon}  | + |\langle\almosteigenfun_p, \almosteigenfun_q\rangle_{\varepsilon} -(1+\lambda^0_{\bar{n}})\delta_{p,q}|\\
     &\leq \|S^{\varepsilon}_p \|_{\varepsilon}\|S^{\varepsilon}_q-\almosteigenfun_q\|_{\varepsilon} +
     \|S^{\varepsilon}_p-\almosteigenfun_p\|_{\varepsilon}\|\almosteigenfun_q \|_{\varepsilon} +O(\varepsilon^{2m})
     \\
     &\leq C(\varepsilon^{2m} + \frac{4}{T}).
     \end{align*} 
%Here, we used the inequality \eqref{vl1} to estimate $\|S^{\varepsilon}_q-\almosteigenfun_q\|_{\varepsilon}$. 
We conclude that for sufficiently large $T$, the columns $\column$ turn out to be almost orthonormalized that is possible only if $K^{\varepsilon}\geq\tau$. Consequently, for $\tau$-multiple eigenvalue $\lambda^0_{\bar{n}}$ there are at least $\tau$ distinct eigenvalues $k^{\varepsilon}_j, \dots,k^{\varepsilon}_{j+\tau-1} $ of the operator $\newspace$ such that 
   \begin{equation}
   \label{almosteig}
   |\almosteigenval_{\bar{n}} - k^{\varepsilon}_j|\leq Tc_n\varepsilon^{\gamma}, \qquad j=n,\dots,n+\tau-1.
   \end{equation}
   \begin{remark}
   \label{proofestapriori}
   The formula \eqref{almosteig} leads to check the inequality \eqref{estieigen}. Indeed, for each eigenvalue $\lambda^0_n$ of the sequence \eqref{seqeigen}, one can associate the eigenvalue $\eigenv_{M(n)}$ such that $\eigenv_{M(n)}\leq \lambda^0_n + C_n\varepsilon^{\gamma}$. %IT FOLLOWS BY \eqref{justest}
    Moreover $M(n_1)<M(n_2)$ if $n_1<n_2$. Consequently $n<M(n)$ and
         \begin{equation*}
         \eigenv_n\leq\eigenv_{M(n)}\leq \lambda^0_n+C_n\varepsilon^{2m}\leq \lambda^0_n + C_n,
         \end{equation*}
   which implies \eqref{estieigen}.
   \end{remark}
\noindent In order to conclude the proof of Theorem \ref{thmmain}, we will check that the eigenvalues $\eigenv_n, \dots, \eigenv_{n+\tau-1}$ of the sequence \eqref{spectrum} verify the estimate \eqref{justasy}. In other words, the equality $\bar{n}=n$ %and $K^{\varepsilon}=\tau$ 
holds true in \eqref{justest}.
 %First, we assume that $K^{\varepsilon}>\tau$. We have proved in Proposition \ref{prop} that each eigenvalue $\eigenv_j$ and the corresponding eigenfunction $(\uone,\uzero)$ of the problem \eqref{uone}-\eqref{bduzero2} converge to $\{\lambda^0, u^0_1, 0\}$, where $\lambda^0$ is the eigenvalue of the problem \eqref{probv_1^0}-\eqref{bcv_1^01} with eigenfunction $u^1_0$. Then the limits $u^1_0$ inherit the linear independence from $\uone,\cdots,\uone$. Since $\lambda^0_n$ is an eigenvalue with multiplicity $\tau$

%In view of the convergence $\eigenv_j\to \lambda^0_{n}$ for $j=m,\cdots,m+K^{\varepsilon}-1$ and 
%To conclude the proof of Theorem \ref{thmmain}, we must check that equality of indices $n=m$ holds in the estimate \eqref{justasy}. \\ 

\noindent Let $\bar{n}$ be some index such that $\lambda^0_{\bar{n}}$ is $\tau$-multiple eigenvalue of the problem \eqref{probv_1^0}-\eqref{bcv_1^01}. If we assume $M(\bar{n}+\tau-1)>\bar{n}+\tau-1$, then there exists an eigenvalue $\lambda^{\varepsilon}_{J^{\varepsilon}}$ with $J^{\varepsilon}\leq M(\bar{n}+\tau-1)$,  such that
      \begin{equation*}
      \lambda^{\varepsilon}_{J^{\varepsilon}}\leq \lambda^0_{\bar{n}+\tau-1} + \varepsilon^{2m}\lambda^\prime_{\bar{n}+\tau-1}+C\varepsilon^{\gamma} < \lambda^0_{\bar{n}+\tau}.
      \end{equation*}
From Lemma \ref{lemma}, the eigenpair $(\lambda^{\varepsilon}_{J^{\varepsilon}}, U^{\varepsilon}_{J^{\varepsilon}})$ converges to eigenelement $(\lambda^{\ast}_{J^0}, U^{\ast})$ of the limit problem \eqref{probv_1^0}-\eqref{bcv_1^01}, where $U^{\ast}$ is orthogonal to  $U^{0}_{1}, \cdots, U^{0}_{\bar{n}+\tau-1}$ in $L^2(\ballone)$. This last claim is invalid, because of the min-max principle (see, e.g. \cite{BS87}) 
    \begin{equation*}
    \lambda^{\ast}_{J^0} =\max_{\substack{E\subset H^1_0(\ballone)\\ dim E=J^{\varepsilon}-1}} \min_{\substack{v\in E^{\perp}\\v\neq 0}} \frac{\|\grad v\|_{L^2(\ballone)}}{\|v\|_{L^2(\ballone)}}
    \end{equation*}
and the inequality $\lambda^{\ast}_{J^0}< \lambda^0_{\bar{n}+\tau}$. Thus, $\bar{n}=n$ and Theorem \ref{thmmain} is proved.

\section{Asymptotic expansion for $\boldsymbol{m< 0}$} 
We briefly describe the behaviour of the eigenpairs of the problem \eqref{uone}-\eqref{bduzero1} for $m<0$. The proof of Theorem \ref{thmmain}  uses the same argument as in Section 3. %We construct the asymptotic expansions %for the solution to the problem \eqref{uone}-\eqref{bduzero2} 
%and we only state the modified version of Theorem \ref{thmmain}, since the proof is performed exactly in the  same way as in Section $3$.\\
We consider an asymptotic expansion for an eigenvalue $\eigenv$ and the corresponding eigenfunction $\{\uzero,\uone\}$ of the form
       \begin{align}
       \eigenv &= \lambda^0 + \varepsilon\lambda^\prime +\cdots, \label{regasyeig}\\
        \uzero (x)&= u^0_0(x) +\varepsilon u^\prime_0(x) +\cdots,\qquad x\in\ballzero,\label{regasymu0}\\
         \uone(x) &= u^0_1(x)+\varepsilon u^\prime_1(x)+\cdots, \qquad x\in\ballone.\label{regasymu1}
        \end{align} 
The formulas \eqref{regasyeig}-\eqref{regasymu0} mean that the eigenpair $(\eigenv,\{\uzero,\uone\})$ is expected to depend on the parameter $\varepsilon$ continuously. By replacing the expansions \eqref{regasyeig}-\eqref{regasymu0} in the problem \eqref{uone}-\eqref{bduzero1} and collecting the coefficients of the same powers of $\varepsilon$, the leading term in \eqref{regasymu0} is a solution to the homogenuous Neumann problem \eqref{pbv_0^0}, i.e. a constant $c_0$. The correction term $u^\prime_0$ defined up to  an additive constant satisfies 
    \begin{align*}
    -\laplace u^\prime_0(x) =0, \hspace{0.3cm}x\in\ballzero,\hspace{1cm}
    \normalderzero u^\prime_0(x) = \normalderzero u_1^0(x), \hspace{0.3cm}x\in\boundzero.
    \end{align*} 
The compatibility condition reads
   \begin{equation}
   \label{compcond}
   \int_{\boundzero} \normalderzero u_1^0 (x)ds_x =0.
   \end{equation}    
The leading terms $\lambda^0_n$ and $u^0_1$ in the ans\"{a}tze \eqref{regasyeig},\eqref{regasymu1} are obtained from the spectral problem
   \begin{align}
   -\laplace u^0_1(x) &= \lambda^0u^0_1(x), \hspace{0.9cm} x\in\ballone,\label{leapbm<0} \\
    \normalderone u^0_1(x) = 0, \hspace{0.3cm}x\in\boundone,\hspace{0.4cm}&\hspace{0.6cm}
    u^0_1(x) = u^0_0(x), \hspace{0.3cm}x\in\boundzero,\label{leabcm<0}
   \end{align} 
along with the integral condition \eqref{compcond}.  To write down the variational formulation of the problem \eqref{leapbm<0}-\eqref{leabcm<0}, we set 
$H^1_{\bullet}(\ballone, \boundzero)$ as the subspace of functions in $H^1(\ballone)$ with a constant trace on the boundary $\boundzero$. For $\varphi\in H^1_{\bullet}(\ballone, \boundzero)$, the Green formula provides
      \begin{align*}
      -\int_{\ballone}\laplace u^0_1(x)\varphi(x)dx &=  \int_{\ballone} \grad u^0_1(x)\grad \varphi(x)dx - \int_{\boundzero} \normalderzero u^0_1(x)\varphi(x)ds_x \\
      &= \lambda^0\int_{\ballone}u^0_1(x)\varphi(x)dx.
      \end{align*}
Since $\varphi$ is a constant on the boundary $\boundzero$, it follows that 
     \begin{equation*}
     \int_{\boundzero} \normalderzero u^0_1(x)\varphi(x)ds_x = \text{const}\int_{\boundzero} \normalderzero u^0_1(x)ds_x =0.
     \end{equation*}
As a consequence, the variational formulation reads: find the eigenvalue $\lambda^0\in\mathbb{R}$ and the corresponding eigenfunction $u^0_1\in H^1_{\bullet}(\ballone, \boundzero)\setminus\{0\} $ such that
   \begin{equation}
   \label{weakformm<0}
   (\grad u^0_1, \grad \varphi)_{\ballone} = \lambda^0(u^0_1, \varphi)_{\ballone}\qquad \forall \varphi\in H^1_{\bullet}(\ballone, \boundzero).
   \end{equation} 
Note that the integral equality \eqref{weakformm<0} implies the condition \eqref{compcond}. The problem \eqref{weakformm<0} admits the eigenvalues sequence \eqref{seqeigen}. The corresponding eigenfunctions $u^0_1$ are orthonormalized in $L^2(\ballone)$.\\
% % % % % % % % % % AGGIUNGERE NELLA TESI % % % % % % % % % % % % % % %
%We can associate to \eqref{weakformm<0} a positive, symmetric and compact operator $\mathcal{A}$ defined on $H^1_{\bullet}(\ballone, \boundzero)$ by
    % \begin{equation}
     %\label{opm<0}
    % \langle\mathcal{A}u,v\rangle= \int_{\ballone} uvdx, \qquad \forall u,v\in H^1_{\bullet}(\ballone, \boundzero).
     %\end{equation}
%Therefore, the spectrum of $\mathcal{A}$ is discrete and the eigenvalues are $(\lambda^0+1)^{-1}$ with $\lambda^0$ an eigenvalue of \eqref{weakformm<0}. As a consequence, the problem \eqref{weakformm<0} admits the eigenvalues sequence \eqref{seqeigen}. The corresponding eigenfunctions $u^0_1$ are orthonormalized in $L^2$-norm. \\ 
The correction term $u^\prime_1$ verifies the problem 
        \begin{align}
        -\laplace u^\prime_1(x) - \lambda^0 u^\prime_1(x) &= \lambda^\prime u^0_1(x), \hspace{0.5cm}x\in\ballone,\label{corprm<0}\\
        \normalderone u^\prime_1(x) =0, \hspace{0.3cm}x\in\boundone,\hspace{0.3cm}&\hspace{0.9cm}
        u^\prime_1(x) = u^\prime_0(x),\hspace{0.3cm} x\in\boundzero.\label{corbcm<0}
        \end{align}
Before computing  $\lambda^\prime$, we investigate term of higher order in the asymptotic \eqref{regasymu0}. We assume that $-2m>1$. The term $u''_0$ of order $\varepsilon^2$ solves the problem 
         \begin{align*}
         -\laplace u''_0(x) =0, \hspace{0.3cm}x\in\ballzero,\hspace{1cm}
          \normalderzero u''_0(x) = \normalderzero u_1^\prime(x), \hspace{0.3cm}x\in\boundzero.
          \end{align*}
The compatibility condition reads as
      \begin{equation}
      \label{com2m1}
      \int_{\boundzero} \normalderzero u_1^\prime(x) ds_x =0.
      \end{equation}    
When $-2m=1$ the compatibility condition becomes inhomegeneous, i.e. 
       \begin{equation}
       \label{com2m2}
       \int_{\boundzero} \normalderzero u_1^\prime(x) ds_x = \lambda^0 c_0|\ballzero|,
       \end{equation} 
since $u''_0$ solves the problem 
          \begin{align}
         -\laplace u''_0(x) =\lambda_0u^0_0(x), \hspace{0.3cm}x\in\ballzero,\hspace{1cm}
          \normalderzero u''_0(x) = \normalderzero u_1^\prime(x), \hspace{0.3cm}x\in\boundzero.\label{u''}
          \end{align}
In the last case, i.e. $-2m<1$, the term $u''_0$ has order $\varepsilon^{-2m+1}$ and it solves the problem \eqref{u''}, yielding the compatibility condition \eqref{com2m2}.\\ In case of simple eigenvalue $\lambda^0_n$ and $-2m>1$, the Fredholm alternative leads to the following expression for the correction term $\lambda^\prime_n$:
      \begin{align}
      \lambda^\prime_n= -(\normalderzero u^0_1, u^\prime_1)_{\boundzero} \label{lambda'1} = -(\normalderzero u^\prime_0, u^\prime_0)_{\boundzero} = - (\grad u^\prime_0, \grad u^\prime_0)_{\ballzero} = -\|\grad u^\prime_0\|^2_{L^2(\ballzero)}.
      \end{align}
If $-2m\leq 1$, the term $\lambda^\prime$ becomes
      \begin{equation*}
      \lambda^\prime= -(\normalderzero u^\prime_1, u^0_1)_{\boundzero}-(\normalderzero u^0_1, u^\prime_1)_{\boundzero}  = -\lambda^0c_0^2|\ballzero| -\|\grad u^\prime_0\|^2_{L^2(\ballzero)},
      \end{equation*}
due to the compatibility condition  \eqref{com2m2}.\\
%In order to remove the dependence on the unknown function $u^\prime_1$, we write more terms in the ansatz for $u^\varepsilon_0$ 
 %   \begin{equation*}
  %  \uzero (x)= u^0_0(x) +\varepsilon u^\prime_0(x)+\varepsilon^{2} u''_0(x) +\cdots,\qquad x\in\ballzero. %\label{regasymu0}
   % \end{equation*}
%where $u''_0$ is the solution to the problem 
 %           \begin{align}
  %           -\laplace u''_0(x) &=0, \hspace{1.5cm}x\in\ballzero,\label{u''}\\
   %          \normalderzero u''_0(x) &= \normalderzero u_1^\prime(x), \hspace{0.4cm}x\in\boundzero.\label{bcu''0}
    %         \end{align} 
%Taking into account the compatibility condition
 %    \begin{equation*}
  %   \int_{\boundzero} \normalderzero u_1^\prime(x) ds_x =0, 
   %  \end{equation*}   
%we find 
 %   \begin{align}
  %  \label{sineig}
   % \lambda^\prime=  -(\normalderzero u^0_1, u^\prime_1)_{\boundzero} = (\normalderzero u^\prime_0, u^\prime_0)_{\boundzero} = - (\grad u^\prime_0, \grad u^\prime_0)_{\ballzero} = -\|\grad u^\prime_0\|^2_{L^2(\ballzero)}.
   % \end{align}    
Suppose that $\lambda^0_n$ is a $\tau$-multiple eigenvalue. As in Section $2$, the leading terms of the asymptotics of the eigenfunctions $u^{\varepsilon}_{1,n}, \dots, u^{\varepsilon}_{1,n+\tau-1}$ are predicted in the form of linear combinations
   \begin{equation*}
    \label{leadmulteig}
   U^0_{1,j}(x) = a^j_nu^0_{1,n}(x) +\dots +a^j_{n+\tau-1}u^0_{1,n+\tau-1}(x), \qquad j=n,\dots, n+\tau-1,
   \end{equation*}
of the eigenfunctions $u^{0}_{1,n}, \dots, u^{0}_{1,n+\tau-1}$ of the limit problem \eqref{leapbm<0}-\eqref{leabcm<0}. The coefficients $a^j_n,\dots,a^j_{n+\tau-1}$ satisfy the orthonormalization condition \eqref{orthcondeigenv}. The first order corrector $U^\prime_{1,j}$ in \eqref{regasymu1} verifies the problem
   \begin{align*}
   -\laplace U^\prime_{1,j}(x)-\lambda^0_nU^\prime_{1,j}(x)&=\lambda^\prime_jU^0_{1,j}(x), \hspace{1.7cm}x\in\ballone,\\
   \normalderone U^\prime_{1,j}(x) =0, \hspace{0.3cm}x\in\boundone, \hspace{0.5cm}&\hspace{0.9cm}
   U^\prime_{1,j}(x)=\sum_{k=n}^{n+\tau-1} a^j_ku^\prime_{0,k}(x), \hspace{0.3cm}x\in\boundzero.
   \end{align*}
In the case $-2m>1$, from the Fredholm alternative and the formula \eqref{com2m1}, we get the $\tau$ compatibility conditions
    \begin{align*}
    \lambda^\prime_j(U^0_{1,j}, u^0_{1,q})_{\ballone} = (\normalderzero u^0_{1,q} , U^\prime_{1,j})_{\boundzero} = \sum_{k=n}^{n+\tau-1}a^j_k(\normalderzero u^\prime_{0,q} ,u^\prime_{0,j})_{\boundzero}= \sum_{k=n}^{n+\tau-1}a^j_k(\grad u^\prime_{0,j},\grad u^\prime_{0,q})_{\ballzero}.
    \end{align*}
It follows that
     \begin{align*}
     \lambda^\prime_ja^j_q = \sum_{k=n}^{n+\tau-1}a^j_k(\grad u^\prime_{0,j},\grad u^\prime_{0,q})_{\ballzero},
     \end{align*}
which can be written in the form of the linear system of $\tau$ algebraic equations 
   \begin{equation}
   \label{min1}
    Ga^j= \lambda^\prime_ja^j, \qquad j=n,\dots, n+\tau-1.
   \end{equation}
Here,  $G$ is the Gram matrix whose entries are 
   \begin{equation*}
   G_{i,j}= (\grad u^\prime_{0,i},\grad u^\prime_{0,j})_{\ballzero}, \qquad i,j=n,\dots,n+\tau-1.
   \end{equation*}
Since $G$ is a symmetric $\tau\times\tau$ matrix, its eigenvalues $\lambda^\prime_n,\dots,\lambda^\prime_{n+\tau-1}$ are real and positive. Indeed  the derivatives $\normalderzero u^0_{1,n}, \dots, \normalderzero u^0_{1,n+\tau-1}$ are linearly independent in $L^2(\boundzero)$. Otherwise, a linear combination 
     \begin{equation*}
     U(x)= \sum_{k=n}^{n+\tau-1} a_iu^0_{1,i}(x), \qquad x\in\ballone,
     \end{equation*}
satisfies the equation $-\laplace U(x) = \lambda^0U(x)$, $x\in\ballone$, and simultaneously two boundary conditions $U(x) = \text{const}$ and $\normalderzero U(x)=0$, $x\in\boundzero$. This is a contraddiction due to the theorem on strong unique continuation (e.g. cf. \cite{Z83}). Hence, $\grad u^\prime_{0,n}, \dots, \grad u^\prime_{0,n+\tau-1}$ are linearly independet in $L^2(\ballzero)^d$ and the matrix $G$ is positive-definite matrix. We emphasize that $\grad u^\prime_{0,i}$, $i=n,\cdots,n+\tau-1$, are defined uniquely, although $u^\prime_{0,i}$ are defined up to a constant.\\If $-2m\leq 1$, the Fredholm alternative and the expression \eqref{com2m2} yield to 
    \begin{align}
    \label{me}
    \lambda^\prime_j = -(\normalderzero U^\prime_{1,j}, u^0_{1,q})_{\boundzero}-(\normalderzero u^0_{1,q} , U^\prime_{1,j})_{\boundzero} = \lambda^0_n(u^0_{0,q})^2|\ballzero| -\sum_{k=n}^{n+\tau-1}a^j_k(\grad u^\prime_{0,j},\grad u^\prime_{0,q})_{\ballzero}.
    \end{align}
As far as the justification procedure is concerned, the estimate \eqref{justasy} of the Theorem \ref{thmmain} for $m<0$ is valid with $\alpha=0$, $\beta=1$, $\gamma=\min\{1-m,2\}$, $\lambda^0_n$ being an eigenvalue of the problem \eqref{leapbm<0}-\eqref{leabcm<0} and $\lambda'_n$ the correction term in \eqref{regasyeig}, given by formula  \eqref{lambda'1} for a simple eigenvalue and  by formulas \eqref{min1}-\eqref{me} for a multiple ones.

\section{Asymptotic expansion for $\boldsymbol{m=0}$}
If $m=0$, the ans\"{a}tze \eqref{regasyeig}-\eqref{regasymu0} are still correct. As a consequence, the problem \eqref{pbv_0^0} is satisfied by the leading term $u^0_0$. %simply become
                  %\begin{align}
                   % \eigenv &= \lambda^0 + \varepsilon\lambda^\prime +\cdots, \label{ansm=0}\\
                    % \uone(x) &= u^0_1(x)+\varepsilon u^\prime_1(x)+\cdots, \qquad x\in\ballone,\label{ansm=01}\\
                     %\uzero (x)&= u^0_0(x) +\varepsilon u^\prime_0(x) +\cdots,\qquad x\in\ballzero.\label{ansm=02}
                     %\end{align} 
The main difference comes from the problem satisfied by the correction term $u^\prime_0$
    \begin{align*}
    -\laplace u^\prime_0(x) = \lambda^0u_0^0(x), \hspace{0.3cm}x\in\ballzero, \hspace{1cm}
     \normalderzero u^\prime_0 (x)= \normalderzero u^0_1(x), \hspace{0.3cm}x\in\boundzero,
     \end{align*}
whose compatibility condition reads as
      \begin{equation}
      \label{constm=0}
      c_0=\frac{1}{\lambda^0|\ballzero|}\int_{\boundzero} \normalderzero u^0_1ds_x.
      \end{equation}
The leading terms $\lambda^0$ and $u^0_1$ in \eqref{regasyeig} and\eqref{regasymu1} solve the problem \eqref{leapbm<0}-\eqref{leabcm<0} along with the integral condition \eqref{constm=0}. Therefore, the variational formulation reads as
    \begin{equation}
    \label{varidenm=0}
    (\grad u^0_1, \grad \varphi)_{\ballone} = \lambda^0 (u^0_1, \varphi)_{\ballzero}+ \lambda^0|\ballzero|\overline{u}^0_1\overline{\varphi},\qquad \forall \varphi\in H^1_{\bullet}(\ballone, \boundzero).
    \end{equation}
Here $\overline{u}$ denotes the constant trace of function $u\in H^1(\ballone)$ on the boundary $\boundzero$.  The problem \eqref{varidenm=0} admits discrete spectrum given by \eqref{seqeigen}.
The correction term $u^\prime_1$ satisfies the problem \eqref{corprm<0}-\eqref{corbcm<0}.
%Since the trace of functions in $H^1_{\bullet}(\ballone, \boundzero)$ can be extended to a constant functions in $H^1(\ballzero)$, the term $\lambda^0|\ballzero|\overline{u}^0_1\overline{\varphi}$ is regarded as the inner product $\int_{\ballzero} UVdx$ in $L^2(\ballzero)$. Therefore, the existence of a discrete spectrum for problem \eqref{varidenm=0} is established in the same way as that for the problem \eqref{weakformm<0} adding to the right hand side of \eqref{opm<0} the scalar product $\int_{\ballzero} UVdx$, with $U, V$ being constant functions. 
The claim of Theorem \ref{thmmain} is still true and the estimate \eqref{justasy} becomes 
     \begin{equation*}
     \label{estm=0}
     |\eigenv-\lambda^0_n-\varepsilon\lambda'_n|\leq C_N\varepsilon^{3/2}.
     \end{equation*}

\section{Asymptotic expansion for $\boldsymbol{m=1/2}$}
This case is discussed in more abstract setting in the textbook \cite[Chapter VII]{SHSP89}, but for the convenience of the reader a simple and independent proof is presented for the problem under consideration. The Helmholtz equation \eqref{uzero}  gets rid of the small parameter $\varepsilon$ 
      \begin{align*}
      %-\laplace\uone(x) &= \eigenv\uone(x), \hspace{1.5cm} x\in\ballone, \label{uonem1/2}\\
      %\normalder\uone(x) &=0, \hspace{2.5cm} x\in \boundone,\label{bduonem1/2}\\
      -\laplace\uzero(x) &= \eigenv\uzero(x), \qquad x\in\ballzero. %\label{uzerom1/2}\\
      %\uzero(x) &= \uone(x), \hspace{1.8cm} x\in \boundzero, \label{bduzero1m1/2}\\
      %\varepsilon^{-1}\normalder\uzero(x) &= \normalder\uone(x), \hspace{1.5cm} x\in \boundzero,  \label{bduzero2m1/2}
      \end{align*}
%The variational setting is 
  %  \begin{equation}
   % \label{weakm=1/2}
    %(\grad \uone, \grad \phi_1)_{\ballone} + \varepsilon^{-1}(\grad\uzero,\phi_0)_{\ballzero} = \eigenv [(\uone, \phi_1)_{\ballone} + \varepsilon^{-1}(\uzero,\phi_0)_{\ballzero}],
    %\end{equation}
%for any $\phi\in H^1(\Omega)$. 
We perfom the replacement \eqref{vzero}, i.e. $v^\epsilon_0(x)=\varepsilon^{-1/2}\uzero(x)$ and $v^\varepsilon_1(x)=\uone(x)$. The asymptotics of eigenpairs $(\eigenv,\{\uzero,\uone\})$ take the form
    \begin{align*}
    \eigenv = \lambda^0 + \varepsilon^{1/2}\lambda^\prime +\cdots, \hspace{0.5cm}&\hspace{1cm}
    v^\epsilon_1(x) = v^0_1(x)+\varepsilon^{1/2} v^\prime_1(x)+\cdots, \qquad x\in\ballone,\\
    v^\epsilon_0 (x)=v^0_0(x) &+\varepsilon^{1/2} v^\prime_0(x) +\cdots,\qquad x\in\ballzero.
    \end{align*}
The essential difference with respect to the other cases is the presence of two spectral limit problems. In fact, the leading term $v^0_0$ is determined from the problem 
     \begin{align}
      -\laplace v^0_0(x) = \lambda^0 v^0_0(x), \hspace{0.3cm}x\in\ballzero,\hspace{1cm}
      \normalderzero v^0_0(x) =0 \hspace{1.5cm}x\in\boundzero.\label{lapm=1/2}
     \end{align}
The leading term $v^0_1$ solves the problem 
     \begin{align}
     -\laplace v^0_1(x) &= \lambda^0v^0_1(x), \hspace{0.9cm} x\in\ballone, \label{limprom12}\\
     \normalderone v^0_1(x) = 0,  \hspace{0.3cm}x\in\boundone,\hspace{0.3cm}&\hspace{1.3cm}
     v^0_1(x) = 0, \hspace{0.3cm}x\in\boundzero.\label{bclimprom12}
     \end{align}
%then we get again the the leading term of asymptotic of $\uzero$ is an eigenfunction of Neumann Laplacian on the bounded domain $\ballzero$. 
 %We emphasize that the eigenfunction $u^0_1$ does not have constant trace on the boundary $\boundone$. As a consequence, the weak formulation of the  problem \eqref{leapbm<0}-\eqref{leabcm<0} is: find the eigenvalue $\lambda^0$ and the eigenfunction $u^0_1 \in \mathring{H}$, $u^0_1\neq 0$ satisfying
    %\begin{equation*}
    %(\grad u^0_1, \grad\varphi)_{\ballone} = \lambda^0 (u^0_1, \varphi)_{\ballone}, \qquad \forall\varphi\in H^1_0(\ballone, \boundzero).
   %\end{equation*} 
%Here, the space $ \mathring{H}$ is defined as
 %  \begin{equation*}
  %  \mathring{H}=\{u\in H^1(\ballone)\hspace{0.02cm}:\hspace{0.02cm} u-\tilde{u}^0_0  \in H^1_0(\ballone, \boundzero)\},
   %\end{equation*}
%where $\tilde{u}^0_0$ is a function in $H^1(\ballone)$ such that $\tilde{u}^0_0 = u^0_0$ on $\boundzero$. 
   % \begin{align}
    %-\laplace u^0_1(x) &= \lambda^0 u^0_1(x), \hspace{0.2cm}x\in\ballone \label{prou^0_0m=1/2}\\
    %\normalder u^0_1(x )&=0\hspace{1.3cm}x\in\boundone\notag\\
    %u^0_1(x) &= u^0_0(x), \hspace{0.3cm}x\in\boundzero.\label{bdu^0_0m=1/2}
    %\end{align}
The problem for the correction terms $v^\prime_0$ is
   \begin{align*}
   -\laplace v^\prime_0(x) -\lambda^0v^\prime_0(x) = \lambda^\prime v^0_0(x), \hspace{0.3cm}x\in\ballzero, \hspace{1cm}
   \normalderzero v^\prime_0(x) = \normalderzero v^0_1(x), \hspace{0.3cm} x\in\boundzero. %\label{cortermbdm=1/2}
   \end{align*}
%If $\lambda^0_n$ is simple, then 
   %\begin{align*}
   %\lambda' = (\partial)
   %\end{align*}
%If $\lambda^0_n$ is a simple eigenvalue, then 
 % \begin{align*}
  %\lambda' = (\normalder u'_0, u^0_0)_{\boundzero} = (\normalder u^0_1, u^0_1)_{\boundzero} =(\grad u^0_1, \grad u^0_1)_{\ballone}.
  %\end{align*}
%If $\lambda^0_n$ we repeat the arguments described in the case $m<0$.
The correction term $v^\prime_1$
is determined by the problem \eqref{pbv'_1}-\eqref{bcv'}.
   % \begin{align*}
   % -\laplace u'_1(x) - \lambda^0 u'_1(x) &= \lambda'u^0_1(x), \hspace{0.3cm}x\in\ballone,\\
   % \normalder u'_1(x) &=0, \hspace{0.3cm} x\in\boundone\\
   % u'_1(x) &= u'_0(x)\hspace{0.4cm}x\in\boundzero.
   %\end{align*}
%If $\lambda^0_n$ is simple, then 
   %\begin{align*}
   %\lambda' = (\normalder u^0_1, u'_1)_{\boundzero} =(\normalder u'_0, u'_0)_{\boundzero}
   %\end{align*}
Owing to the two limit problems, the procedure made for the convergence theorem must be slightly modified. We explain it briefly. \\According to the convergence \eqref{conveigenva} of eigenvalues $\eigenv_n$, the weak formulation \eqref{weakform} and the normalization condition \eqref{normalcond} with $m=1/2$, %becomes
    %\begin{equation*}
    %\|u^{\varepsilon}_{1,n}\|_{L^2(\ballone)} ^2+\varepsilon^{-1}\|u^{\varepsilon}_{0,n}\|^2_{L^2(\ballzero)} =1
    %\end{equation*}
%Then from the weak formulation of the problem \eqref{uone}-\eqref{bduzero2},
we deduce 
   \begin{equation*}
   \|\grad v^{\varepsilon}_{1,n}\|_{L^2(\ballone)} ^2+\varepsilon^{-1}\|\grad v^{\varepsilon}_{0,n}\|^2_{L^2(\ballzero)} \leq C_n.
   \end{equation*}
As in Section $3.1.1$, $v^{\varepsilon}_{0,n}$ converges to zero strongly in $H^1(\Omega_0)$ and hence in
$L^2(\Omega_0)$, while $v^{\varepsilon}_{1,n}$
converges to some $v^0_{1,\bar{n}}$ weakly in $H^1(\Omega_1)$ and strongly in $L^2(\ballone)$. If $v^0_{1,\bar{n}}\neq 0$, the continuity of trace operator ensures that $v^{\varepsilon}_{0,n}$ converges to $0$ in $L^{2}(\boundzero)$. Then the boundary condition \eqref{bduzero1} yields the strong convergence of  $v^{\varepsilon}_{1,n}$ to $0$ in $L^2(\boundzero)$,
i.e. $v^0_{1,\bar{n}}\in H^1_0(\Omega_1, \boundzero)$. Using the same arguments as in Section $3.1.1$, we deduce that  the leading terms $\lambda^0_{\bar{n}}$ and $v^0_{1,\bar{n}}$, with $v^0_{1,\bar{n}}\neq 0$, are characterized as the eigenelements of the spectral problem  \eqref{limprom12}-\eqref{bclimprom12}.\\
Assume, now, that $v^0_{1,\bar{n}}=0$. The previous arguments fail and we introduce the new normalization condition
     \begin{equation}
     \label{newnormcond}
     \|v^{\varepsilon}_{1,n}\|_{L^2(\ballone)} ^2+\varepsilon^{-1}\|v^{\varepsilon}_{0,n}\|^2_{L^2(\ballzero)} =\varepsilon^{-1}.
     \end{equation}
The weak formulation \eqref{weakform} implies the bound
    \begin{equation}
    \label{cond1}
    \|\grad v^{\varepsilon}_{1,n}\|_{L^2(\ballone)} ^2+\varepsilon^{-1}\|\grad v^{\varepsilon}_{0,n}\|^2_{L^2(\ballzero)} \leq C_n\varepsilon^{-1}.
    \end{equation}     
Multiplying the inequalities \eqref{newnormcond} and \eqref{cond1} by $\varepsilon$, the quantities $\|v^{\varepsilon}_{0,n}\|_{H^{1}(\ballzero)}$ and $\|\grad v^{\varepsilon}_{0,n}\|_{H^{1}(\ballzero)}$ are bounded and then  $v^{\varepsilon}_{0,n}$ converges weakly in $H^{1}(\ballzero)$ and strongly to $L^{2}(\ballzero)$ to some function $v^{0}_{0,\bar{n}}$. %Thus, $u^{\varepsilon}_{0,n}$ converges to some $\hat{u}^0_{0,n}$ weakly in $H^1(\ballzero)$ and strongly in $L^2(\ballzero)$.
Moreover, the trace of $v^{\varepsilon}_{0,n}$ converges to the trace of $v^0_{0,\bar{n}}$
in $L^2(\boundzero)$. Finally, the limit passage as $\varepsilon\to 0$ in the weak formulation \eqref{weakform} leads to characterize $v^0_{0,\bar{n}}$ as the eigenfunction with associated eigenvalue $\lambda^0_{\bar{n}}$ of the problem \eqref{lapm=1/2} and the eigenfunctions $v^0_{0,\bar{n}}$ are normalized in $L^2(\ballzero)$. Indeed bearing in mind that $\lambda^0_{\bar{n}}$ does not belong to the spectrum of the problem \eqref{pbv'_1}-\eqref{bcv'} and the convergence \eqref{conveigenva}, for small $\varepsilon>0$  $\eigenv_n$  is not an eigenvalue of the problem 
        \begin{align*}
         -\laplace v^{\varepsilon}_{1,n} (x) &= \lambda^{\varepsilon}_nv^{\varepsilon}_{1,n}(x), \hspace{0.5cm} x\in\ballone,\\
         \normalderone v^{\varepsilon}_{1,n}(x)=0,\hspace{0.3cm} x\in\boundone,\hspace{0.3cm}&\hspace{1cm}
         v^{\varepsilon}_{1,n}(x)=0, \hspace{0.3cm} x\in\boundzero.
         \end{align*} 
As a consequence, we have 
     \begin{equation}
     \label{boundness}
     \|v^{\varepsilon}_{1,n}\|_{H^1(\ballone)}\leq c\|v^{\varepsilon}_{0,n}\|_{H^{1/2}(\boundzero)}\leq c\|v^{\varepsilon}_{0,n}\|_{H^{1}(\ballzero)}\leq c.
     \end{equation} 
The inequalities  \eqref{newnormcond} and \eqref{boundness} lead to check the normalization condition of the eigenfunction $v^0_{0,\bar{n}}$. The Theorem \ref{thmmain} is still valid with the estimate
        \begin{equation*}
            \label{estm=0}
            |\eigenv-\lambda^0_n-\varepsilon^{1/2}\lambda'_n|\leq C_N\varepsilon.
            \end{equation*}

% % % % %AGGIUNGERE ALLA TESI % % % % % %
%As far as Lemma about near eigenvalues and eigenfunctions concern, the approximate eigenfunction takes the form 
 %     \begin{equation*}
 %     \almosteigenfun_n = ( v^0_{1,n}+\varepsilon^{1/2} v'_{1,n}  ,\hspace{0.2cm} v^0_{0,n} +\varepsilon^{1/2} v'_{0,n}    ).
  %    \end{equation*}
%which implies the estimate \eqref{estm=0}. 

% NO AGGIUNGEREChoosing the test functions $\varphi_1\in H^1_0(\ballone, \boundzero)$ and $\varphi_0 \equiv 0$ in $\ballzero$ in the integral identity \eqref{weakform},  we performe the limit passage as $\varepsilon\to 0$, obtaining 
 %   \begin{equation*}
    %(\grad u^0_1, \grad \phi_1)_{\ballone}  = \lambda^0 (u^0_1, \phi_1)_{\ballone},
    %\end{equation*}
%which can read as the weak formulation of problem \eqref{leapbm<0}-\eqref{leabcm<0}. On the other hand,  we multiply the weak formulation \eqref{weakform} by the small parameter $\varepsilon$ and choose an arbitray function $\varphi\in H^1(\Omega)$. 
   %\begin{equation*}
    %\varepsilon(\grad \uone, \grad \varphi_1)_{\ballone} + (\grad\uzero,\varphi_0)_{\ballzero} = \eigenv [\varepsilon(\uone, \varphi_1)_{\ballone} + (\uzero,\varphi_0)_{\ballzero}].
   %\end{equation*}
%The limit passage  as $\varepsilon\to 0$ yields to the following integral idenity
 %  \begin{align*}
  % (\grad u^0_0,\varphi_0)_{\ballzero} = \lambda^0 (u^0_0,\varphi_0)_{\ballzero},
   %\end{align*}
%which is the weak formulation of the spectral problem with Neumann boundary condition \eqref{lapm=1/2}-\eqref{bdm=1/2}.

\section{Asymptotic expansion for $\boldsymbol{m>1/2}$}

We postulate the asymptotic expansions for the eigenvalue $\eigenv$ 
    \begin{equation}
    \label{eigenm>12}
    \eigenv = \varepsilon^{2m-1}\lambda^0+ \varepsilon^{2m}\lambda^\prime+\cdots.
    \end{equation}
For the corresponding eigenfunctions $\{ \uzero,\uone\}$ we consider an asymptotic expansion of the form
     \begin{align}
     \uzero&=u^0_0+\varepsilon u^\prime_0+\cdots,\hspace{2.4cm}x\in\ballzero,\label{u0m>1/2}\\
     \uone &= u^0_1+\varepsilon^{\min\{1,2m-1\}}u_1^\prime+\cdots, \qquad x\in\ballone.\label{u1m>1/2} 
     \end{align}
%In the Helmholtz equation \eqref{uzero}, $\varepsilon^{1-2m}$ gets bigger as $\varepsilon\to 0$ and hence this suggests to modify the asymptotic expansion for $\eigenv$. Note that the minmax principle gives the equalities
 %    \begin{equation*}
  %   \eigenv_n = \min_{\substack{E_n\subset H^1(\Omega) \\ \dim E_n = n+1}} \max_{\substack{\{U,u\}\in E_n\\ \{U,u\}\neq 0 }} \frac{\int_{\ballone}|\grad U|^2dx+\varepsilon^{-1} \int_{\ballzero}|\grad u|^2dx}{\int_{\ballone}|U|^2dx  + \varepsilon^{-2m}\int_{\ballzero}|u|^2dx   }.
   %  \end{equation*}
%Then, taking the particular subspace of $H^1(\Omega)$, $E_n=\spanned \{\{u^0_{1,0},u^0_{0,0}\}, \{u^0_{1,0},u^0_{0,1}\}, \cdots, \{u^0_{1,0},u^0_{0,n}\}\}$, where $\{u^0_{1,i}, u^0_{0,i}\}$ are respectively the solution of the problems \eqref{seclimpbm>1/2}-\eqref{seclimbcm>1/2} and \eqref{limprobm>1/2}-\eqref{limbcm>1/2}. Here $\spanned\{\cdot\}$ stands for a linear space. Owing to the choice of $E_n$, the minmax principle applied to the eigenvalues $\mu_0$ of the problem  \eqref{limprobm>1/2}-\eqref{limbcm>1/2} and the variational formulation of the problem \eqref{seclimpbm>1/2}-\eqref{seclimbcm>1/2}, we deduce 
     %\begin{equation*}
     %\eigenv_n\leq C\max_{\substack{\{U,u\}\in E_n\\ \{U,u\}\neq 0 }} \frac{\varepsilon^{-1} \int_{\ballzero}|\grad u|^2dx}{\int_{\ballone}|U|^2dx  + \varepsilon^{-2m}\int_{\ballzero}|u|^2dx } \leq C \varepsilon^{2m-1}\mu_0.
     %\end{equation*}

       % \begin{equation}
        %\varepsilon^{1-2m}\lambda^{\varepsilon} = \mu^{\varepsilon}.
        \ %end{equation*}
%Then, equations \eqref{uone}, \eqref{uzero} turn into
 %      \begin{align*}
  %     -\laplace\uone(x) &= \varepsilon^{2m-1}\mu^{\varepsilon}\uone(x),\hspace{0.5cm}x\in\ballone,\\
   %     -\laplace\uzero(x) &= \mu^{\varepsilon}\uzero(x), \hspace{1.4cm} x\in\ballzero,
    %   \end{align*}
%while the boundary condition on $\boundone$ and trasmission conditions on $\boundzero$ remain unchanged. We look for ans\"{a}tze for the eigenfunctions $\{\uone, \uzero\}$ in the form \eqref{regasymu1}, \eqref{regasymu0} and the asymptotics of eigenvalue $\mu^{\varepsilon}$ becomes
       % \begin{equation*}
        %\label{asylamm01}
        %\eigenv = \varepsilon^{2m-1}\mu^0 + \varepsilon^{2m}\mu' +\cdots.
        %\end{equation*}
        
\noindent Using the same procedure as in the other cases, we find that the leading terms $\lambda^0$, $u^0_0$ in \eqref{eigenm>12}, \eqref{u0m>1/2} are characterized as the solution to the spectral problem 
     \begin{equation}
       -\laplace u^0_0(x)=\lambda^0u^0_0(x), \hspace{0.3cm}x\in\ballzero,\hspace{1cm}
       \normalderzero u^0_0(x) =0, \hspace{0.3cm} x\in\boundzero.\label{limprobm>1/2}
     \end{equation} 
The problem \eqref{limprobm>1/2} in the Sobolev space $H^1(\ballzero)$ has a discrete spectrum
    \begin{equation*}
    0=\lambda^0_1<\lambda^0_2\leq\cdots\leq\lambda^0_n\leq\cdots\to\infty
    \end{equation*}
and the corresponding eigenfunctions $u^0_{0,n}$ are subject to the orthonormalization condition in $L^2(\ballzero)$. The leading term $u^0_1$ in \eqref{u1m>1/2} is defined as a unique solution %in the space $\ring{H}$ 
of the problem
   \begin{align*}
   -\laplace u^0_1(x) &= 0, \hspace{1.1cm}x\in\ballone,\\ %\label{seclimpbm>1/2}\\
   \normalderone u^0_1(x) =0, \hspace{0.3cm}x\in\boundone, \hspace{0.5cm}&\hspace{1.1cm}
   u^0_1(x) = u^0_0(x)\hspace{0.3cm}x\in\boundzero.
   \end{align*}
%Then $u^0_1$ is an harmonic function in the annulus $\ballone$ and this problem admits also a unique solution (see "mixed boundary value problem of Laplce equation in a bounded Lipschitz domain" TongKeun Chang, Hi Jun Choe).
If $\min\{1,2m-1\}= 2m-1$, the problem for the correction term $u^\prime
_1$ in \eqref{u1m>1/2} is  
   \begin{align*}
   -\laplace u^\prime_1(x)&=\lambda^0u^0_1(x), \hspace{0.8cm}x\in\ballone,\\
   \normalderone u^\prime_1(x) =0,\hspace{0.3cm}x\in\boundone,\hspace{0.4cm}&\hspace{1.1cm}
   u^\prime_1(x) = 0, \hspace{0.3cm}x\in\boundzero.
   \end{align*} 
If $\min\{1,2m-1\}= 1$, the correction term $u^\prime
_1$ is characterized as the solution to the problem 
     \begin{align*}
     -\laplace u^\prime_1(x)&=0, \hspace{1.2cm}x\in\ballone,\\
     \normalderone u^\prime_1(x) =0,\hspace{0.3cm}x\in\boundone,\hspace{0.4cm}&\hspace{0.9cm}
     u^\prime_1(x) = u^\prime_0(x), \hspace{0.3cm}x\in\boundzero.
     \end{align*}
We point out that when $m=1$, the problem satisfied by $u^\prime_1$ turns into 
      \begin{align*}
      -\laplace u^\prime_1(x)&=\lambda^0u^0_1, \hspace{0.7cm}x\in\ballone,\\
       \normalderone u^\prime_1(x) =0,\hspace{0.3cm}x\in\boundone,\hspace{0.4cm}&\hspace{0.9cm}
      u^\prime_1(x) = u^\prime_0, \hspace{0.3cm}x\in\boundzero.
      \end{align*}
The correction term $u^\prime_0$ in the asymptotic expansion \eqref{u0m>1/2} is determined from the problem
    \begin{align}
    -\laplace u^\prime_0(x) -\lambda^0 u^\prime_0(x) =\lambda^\prime u^0_0(x), \hspace{0.3cm} x\in\ballzero,\hspace{1cm}
    \normalder u^\prime_0(x) = \normalder u^0_1(x), \hspace{0.3cm} x\in\boundzero. \label{pbmu}
    \end{align}
The compatibility condition in the problem \eqref{pbmu} provides the correction term $\lambda^\prime$. Indeed, if the eigenvalue $\lambda^0_n$ is simple then we get
   \begin{equation}
   \label{m>12simple}
   \lambda^\prime = -\|\grad u^0_1\|^2_{L^2(\ballone)}.
   \end{equation}
Assume now that the eigenvalue $\lambda^0_n$ has multiplicity $\tau>1$, i.e. $\lambda^0_{n-1}<\lambda^0_{n} = \cdots = \lambda^0_{n+\tau-1}<\lambda^0_{n+\tau}$. The leading term in the expansions  \eqref{u0m>1/2} are predicted in the form of linear combinations of the eigenfunctions $u^0_{0,n},\dots,u^0_{0,n+\tau-1}$
   \begin{equation*}
    U^0_{0,j} (x) = a^j_n u^0_{0,n}(x) + \dots +a^j_{n+\tau-1} u^0_{0,n+\tau-1}(x), \qquad j=n,\dots,n+\tau-1.
   \end{equation*}
Therefore, $U^\prime_{0,j}$ is the solution to the problem 
   \begin{align*}
   -\laplace U^\prime_{0,j} (x) -\lambda^0_n U^\prime_{0,j}(x) =\lambda^\prime_j U^0_{0,j}(x), \hspace{0.3cm} x\in\ballzero,\hspace{0.5cm}
   \normalder U^\prime_{0,j}(x) = \sum_{k=n}^{n+\tau-1} a^j_k\normalderzero u^0_{1,k}(x), \hspace{0.2cm}x\in\boundzero.
   \end{align*}
According to the Fredhom alternative, the $\tau$ compatibility conditions are
   \begin{align}
   \lambda^\prime_j (U^0_{0,j}, u^0_{0,p})_{\ballzero} &= (\normalderzero U^\prime_{0,j},u^0_{0,p})_{\boundzero} = \sum_{k=n}^{n+\tau-1}a^j_n (\normalderzero u^0_{1,k}, u^0_{0,p})_{\boundzero}\notag\\
   & = \sum_{k=n}^{n+\tau-1}a^j_n (\grad u^0_{1,k}, \grad u^0_{0,p})_{\ballone}, \hspace{0.4cm} j=n,\dots,n+\tau-1\label{m>12mul}
   \end{align}
The previous relation can be written as 
the formula \eqref{min1} with a different Gram matrix. In other words, the $\tau$ correction terms are the eigenvalues of the Gram matrix $G$ whose entries are given by
   \begin{equation*}
   G_{i.j} = (\grad u^0_{1,i}, \grad u^0_{1,j})_{\ballone}, \qquad i,j=n,\dots, n+\tau-1,
   \end{equation*}
with $a_n^j$ being the corresponding eigenvectors. The estimate \eqref{justasy} of Theorem \ref{thmmain} holds with $\alpha=2m-1$, $\beta=m$, $\gamma= \min\{4m-1,1+m\}$ for $m\in (1/2,1)$ and $\gamma=2m+1$ for $m\geq1$, $\lambda^0_n$ being the eigenvalue of the problem \eqref{limprobm>1/2} and  $\lambda^\prime_n$ being the correction term, given by formula \eqref{m>12simple} if $\lambda^0_n$ is a simple eigenvalue and \eqref{m>12mul} if $\lambda^0_n$ is a multiple one.

\section{Kissing domains in $\boldsymbol{\mathbb{R}^2}$}
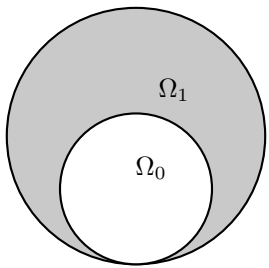
\begin{figure}
\centering
\begin{tikzpicture}
\draw[fill=mygray, thick] (0,0)circle(1.7)node[above, xshift=0.5cm, yshift=9.4]{$\Omega_1$};
\draw [fill=white, thick] (0,-0.7) circle(1)node[above, xshift=0.2cm, yshift=0]{$\Omega_0$}; %node[below, yshift=0.2cm]{$\Omega_0$};
%\draw (0,0) arc (-90:-25:3cm) node[anchor=east] {$\Gamma_0$};
%\draw [<-] (0.9,0.1)--(2,-1) node[near end, below]{$\Omega$};
\end{tikzpicture}
\caption{Kissing domains}\label{fig:kissingdomains}
\end{figure}
A distinguishing feature of the stiff Neumann problem is that all asymptotic forms derived and justified in previuous sections, are preserved when $\dist (\boundzero,\boundone) \to 0$, i.e in the limit the core $\ballzero$ touches the exterior boundary $\boundone$, cf. Fig. \ref{fig:kissingdomains}, so that $\ballzero$ and $\Omega$ form the interior kiss of two domains. This peculiar conclusion is certainly based on the exterior Neumann condition \eqref{Neumann} and the prevaling stiffness of th annulus $\ballone$. Changing particular details in problem's statement may quit the above-mentioned limit passage: in Section 8.4 we will discuss a serious issue in the case when the Neumann condition \eqref{Neumann} is replaced with the homogeneous Dirichlet one. In this way, in many cases the Dirichlet problem of the type \eqref{uone}-\eqref{bduzero1} remains a fully open question. \\The performed asymptotic analysis demonstrates that in all situations under consideration the limit problem in the cuspidal annulus reads
      \begin{align}
      -\laplace u(x) &=\lambda u(x), \qquad x\in\ballone \label{kisspb}\\
      \normalderone u(x)=0, \hspace{0.3cm}x\in\boundone\setminus\cusp\hspace{0.5cm} & \hspace{1.1cm} u(x)=g(x), \hspace{0.3cm} x\in\boundzero\setminus\cusp,\label{kissbc}
      \end{align}
where $\lambda\geq 0$ and $g=0$ or $g=\text{const}$ on the boundary $\boundzero$. Denoting by $G\in H^1(\mathbb{R}^d\setminus\ballzero)$ an extension of $g$ onto the exterior of $\ballzero$, the variational formulation of the problem \eqref{kisspb}-\eqref{kissbc} reads, cf. \cite{L85}: find $u\in H^1(\ballone)$ such that $u-G\in H^1(\ballone; \boundzero)$ and the following integral identity is valid 
     \begin{equation}
     \label{varide}
     (\grad u, \grad v)_{\ballone} = \lambda (u,v)_{\ballone} \qquad\forall v\in  \mathbf{H},
     \end{equation}
with $\mathbf{H} = H^1_0(\ballone; \boundzero)$ if $g=0$ on $\boundzero$ or $\mathbf{H} = H^1_{\bullet}(\ballone, \boundzero)$ if $g=\text{const}$ on $\boundzero$
Due to the Dirichlet condition on $\boundzero$, the space $ H^1_0(\ballone; \boundzero)$ is compactly embedded into $ L^2(\ballone)$ \footnote{This fact is true for $ H^1(\ballone)$ as well, see, e.g. \cite{MP97}}. Hence, all necessary properties of the problem \eqref{kisspb}-\eqref{kissbc} are kept in the cuspidal domain $\ballone$ and these allow us to repeat with easily predictable modifications our calculations and argumentation in previous sections, to conclude the analog of the Theorem \ref{thmmain} in the case of the tounching boundaries $\boundzero$ and $\boundone$. \\ In the sequel we will describe the asymptotic behaviour as $x\to\cusp$ of solutions to the mixed boundary value problem \eqref{kisspb}-\eqref{kissbc} and a similar Dirichlet problem in $\ballone$. However, to reduce cumbersome and long computations, we deal with the 2D case only while a needed modification for multi-dimensional cases can be found in the papers \cite{N94, NST09, NT11, NT18} and others.

%In this section we consider the problem \eqref{uone}-\eqref{bduzero2} in a union of domains whose boundaries loose smoothness. More precisely, the bounded domain $\Omega$ is in the plane $\mathbb{R}^2$ and consists of two disks $\Omega_1$ and $\Omega_0$ forming a cuspidal point (see Fig.\ref{fig:kissingdomains}). We assume that $\Omega_1$ and $\Omega_0$ touch each other at the origin $\cusp$ of the Cartesian coordinate system $x=(x_1, x_2)$ and the boundary $\partial \ballone = \boundone\cup\boundzero$ is smooth outside any neighbourhood of $\cusp$.  We will provide the asymptotic formula for the eigenfunctions of the limit problem \eqref{leapbm<0}-\eqref{leabcm<0} near the singular point $\cusp$ of the boundary. Then the asymptotic expansions for the eigenpairs $\{\eigenv, u^{\varepsilon}\}$ and the justification procedure given in the Sections $2-4$ are still valid.

\subsection{Asymptotics of solutions at the cusp in Neumann case}
We consider the spectral problem  \eqref{kisspb}-\eqref{kissbc}, where $g=c_0$ on $\boundzero\setminus\cusp$ and $c_0$ is an arbitrary constant. Set $R_0, R_1$  the radii of the disks $\ballzero, \ballone$ respectively such that $R_0<R_1$. The boundaries $\boundzero$ and $\boundone$ are described by 
  \begin{equation}
  \label{Hfunc}
   H_i(x_1) = \frac{|x_1|^2}{2R_i} + O(|x_1|^4), \qquad i=0,1.
  \end{equation}
The thickness is defined as  $H(x_1)=H_0(x_1) - H_1(x_1)$.  We write down the representation
     \begin{equation}
     \label{zeroterm}
     u(x)=c_0+\cdots, \qquad x\rightarrow\cusp,
     \end{equation}
where the dots denote the lower-order terms. The distinguished asymptotic term on the right-hand side of \eqref{zeroterm} satisfies the boundary conditions \eqref{kissbc} but generates the residual 
    \begin{equation*}
    \label{residual}
    \lambda c_0 + \cdots 
    \end{equation*}
in the differential equation \eqref{kisspb}. Then, we introduce a new term $\firstterm$ in \eqref{zeroterm},  involving the stretched coordinate
      \begin{equation*}
      \label{strectcoord}
      \eta = \frac{x_2-H_{1}(x_1)}{H(x_1)}\in (0,1).
      \end{equation*}
Then the asymptotic \eqref{zeroterm} turns into 
     \begin{equation}
     \label{ansatz}
     u(x) = c_0 + \firstterm +\cdots, \qquad x\rightarrow\cusp.
     \end{equation}
In order to rewrite \eqref{kisspb} in the new variables $(x_1,\eta)$, we evalute
      \begin{align}
      \frac{\partial}{\partial x_2} &= \frac{\partial \eta}{\partial x_2}\frac{\partial}{\partial \eta} = \frac{1}{H(x_1)}\frac{\partial}{\partial \eta},\hspace{2.7cm}  \frac{\partial^2}{\partial x_2^2} = \frac{1}{H(x_1)^2}\frac{\partial^2}{\partial \eta^2},\label{dereta}\\
      \frac{\partial}{\partial x_1} &= \frac{\partial}{\partial x_1}-H(x_1)^{-1}(H'_1(x_1)+\eta H'(x_1))\frac{\partial}{\partial\eta}, \notag\\ 
      \frac{\partial^2}{\partial x_2^2} &=\left(\frac{\partial}{\partial x_1} -H(x_1)^{-1}(H'_1(x_1)+\eta H'(x_1))\frac{\partial}{\partial\eta}\right)^2\notag \\
      &\quad+ \left(\frac{2H'(x_1)H'_1(x_1) + 2(H'(x_1))^2\eta-H''_1(x_1)H(x_1) - H''(x_1)H(x_1)\eta}{H(x_1)^2}\right)\frac{\partial}{\partial\eta},\label{derx_1}
      \end{align}
where $H'(x_1) = \frac{dH(x_1)}{dx_1}$. Owing to \eqref{Hfunc}, \eqref{dereta}, \eqref{derx_1}, the Laplace operator $\Delta_{(x_1,x_2)}$ in the new variables $(x_1,\eta)$ is written as
   \begin{align}
   \Delta_{(x_1,\eta)} &= %\frac{1}{H(x_1)^2}\frac{\partial^2}{\partial \eta^2} + \biggl(  \frac{\partial}{\partial x_1}-H(x_1)^{-1}(H'_1(x_1)+\eta H'(x_1))\frac{\partial}{\partial\eta}  \biggr)^2\notag\\
   %&= \frac{1}{H^p(x_1)^2}\frac{\partial^2}{\partial \eta^2} + \biggl(  \frac{\partial}{\partial x_1}-H^p(x_1)^{-1}((H^p)'_1(x_1)+\eta (H^p)'(x_1))\frac{\partial}{\partial\eta}  \biggr)^2+\cdots,\\
    \frac{1}{H^p(x_1)^2}\frac{\partial^2}{\partial \eta^2} + \sum_{j=1}^{\infty} L_j(x_1, \eta, \frac{\partial}{\partial x_1},\frac{\partial}{\partial\eta}),\label{lapkissdom}
   \end{align}
where we replaced the thickness function $H(x_1)$ with its principal part
    \begin{equation*}
    H^p(x_1)=\frac{1}{2}\left(\frac{1}{R_0}-\frac{1}{R_1}\right)|x_1|^2.
    \end{equation*}
The normal derivative on the lower boundary $\boundone$ can be written as
        \begin{align}
         \normalderone &= \frac{1}{(1+|H'_{1}(x_1)|^2)^{1/2}}\biggl(\frac{\partial}{\partial x_2}-H'_{1}(x_1)\frac{\partial}{\partial x_1}\biggr)\notag \\
         &= \frac{1}{(1+|H'_{1}(x_1)|^2)^{1/2}}\biggl(\frac{1}{H(x_1)}\frac{\partial}{\partial \eta} -H'_{1}(x_1) \frac{\partial}{\partial x_1}
         +H'_{1}(x_1)H(x_1)^{-1}(H'_1(x_1)+\eta H'(x_1))\frac{\partial}{\partial\eta}\biggr).\label{normderkiss}
        \end{align}
In view of \eqref{lapkissdom} and \eqref{normderkiss}, we insert the expansion ansatz \eqref{ansatz} into the problem \eqref{kisspb}-\eqref{kissbc}, obtaining the problem 
%In order to compensate the residual \eqref{residual}, we assume the following equation %  According to \eqref{Hfunc} and \eqref{dereta}, we obtain the problem % obtaining
     %\begin{equation*}
     %\label{discrepances}
     %-\frac{\partial^2}{\partial x_1^2}\firstterm -\frac{1}{H(x_1)^2} \frac{\partial^2}{\partial \eta^2}\firstterm+\cdots = \lambda u(\cusp) + \lambda\firstterm+\cdots.
     %\end{equation*}
     \begin{align}
     -\frac{1}{H^p(x_1)^2} \frac{\partial^2}{\partial \eta^2}\firstterm & = \lambda c_0, \label{pbfirst}\\
     \frac{\partial}{\partial \eta}\firstterm_{\big| \eta = 0} =0,\notag \hspace{1.4cm}&\hspace{0.9cm}
     \firstterm_{\big| \eta = 1} = 0.\label{bdcondU_0}
     \end{align} 
By a direct computation, the solution $\mathcal{U}_1$ is given by
     \begin{equation}
     \label{explfirs}
     \mathcal{U}_1(x_1,x_2) = -\frac{\lambda c_0}{2}[x^2_2-2H_1^p(x_1)(x_2+H_0^p(x_1)) -H^p_0(x_1)^2   ],  
     \end{equation}
where $H_i^p(x_1)=|x_1|^2/(2R_i)$ denotes the principal part of $H_i(x_1)$, $i=0,1$. Note that the first-order correction term $\mathcal{U}_1(x_1,x_2) $ is of order $|x_1|^4$. Iterating this procedure, we are able to construct the formal infinite series of the eigenfunction $u$ of the problem \eqref{kisspb}-\eqref{kissbc} %In fact the leading term $u(\cusp)=c_0$ satisfies the boundary conditions \eqref{bdtouchbal}-\eqref{bdtouchbal1}, leaving discrepancy \eqref{residual} in the Helmholtz equation \eqref{laptouchbal}. To compensate it, we construct the formal series
    \begin{equation}
    \label{formseri}
    u(x)=c_0+\sum_{j=1}^{\infty} \mathcal{U}_j(x),
    \end{equation}
containing the already chosen main term \eqref{explfirs}.
Keeping in mind the decomposition \eqref{lapkissdom}, \eqref{normderkiss} and replacing the eigenfunction $u$ with its formal series \eqref{formseri} into the equation \eqref{kisspb}, we find that the term $\mathcal{U}_2$ is solution of the problem 
    \begin{align*}
    -\frac{1}{H^p(x_1)^2}\frac{\partial^2}{\partial\eta^2}\mathcal{U}_2(x_1,\eta) &=  L_1(x_1, \eta, \frac{\partial}{\partial x_1},\frac{\partial}{\partial\eta},\lambda)\firstterm, \\
   \partial_\eta\mathcal{U}_2(x_1,\eta)_{\big| \eta = 0} =N_1(\frac{\partial}{\partial x_1},\frac{\partial}{\partial\eta})\mathcal{U}_1(x_1,\eta)_{\big| \eta = 0}, \hspace{0.1cm}&\hspace{1.3cm}
   \mathcal{U}_2(x_1, \eta)_{\big| \eta = 1} = 0,
    \end{align*}
where
   \begin{align*}
    L_1(x_1, \eta, \frac{\partial}{\partial x_1},\frac{\partial}{\partial\eta}, \lambda) &= \biggl(  \frac{\partial}{\partial x_1}-H^p(x_1)^{-1}((H^p_1)'(x_1)+\eta (H^p)'(x_1))\frac{\partial}{\partial\eta}  \biggr)^2,\\
    &\quad +\biggl(\frac{2(H^p(x_1))'(H^p_1(x_1))' + 2(H'^p(x_1))^2\eta}{H^p(x_1)^2}\\
    &\quad -\frac{(H^p_1)''(x_1)H^p(x_1) + (H^p(x_1))''H^p(x_1)\eta}{H^p(x_1)^2}\biggr)\frac{\partial}{\partial\eta}+\frac{|x_1|^6}{8R^4}\lambda c_0\\
    N_1(\frac{\partial}{\partial x_1},\frac{\partial}{\partial\eta}) &= (H^p_1)'(x_1)\left( \frac{\partial}{\partial x_1} + \frac{(H^p_1)'(x_1)}{(H^p)'(x_1)}\frac{\partial}{\partial\eta}\right).
   \end{align*}
The other terms of the series \eqref{formseri} are determined by the problems
     \begin{align*}
     -\frac{1}{H^p(x_1)}\partial_{\eta}\mathcal{U}_j(x_1,\eta) &= L_{j-1}(x_1, \eta, \frac{\partial}{\partial x_1},\frac{\partial}{\partial\eta},\lambda)\mathcal{U}_{j-1}(x_1,\eta) + \lambda \mathcal{U}_{j-2}(x_1,\eta), \qquad j=3, 4, \cdots,\\
     \partial_\eta\mathcal{U}_j(x_1,\eta)_{\big| \eta = 0} &= N_{j-1}(0, \frac{\partial}{\partial x_1},\frac{\partial}{\partial\eta})\mathcal{U}_{j-1}(x_1,\eta)_{\big| \eta = 0}, \hspace{1cm}
     \mathcal{U}_j(x_1, \eta)_{\big| \eta = 1} = 0.
     \end{align*}
We point out that the terms $\mathcal{U}_j$, $j=2,3,\cdots$, of the series \eqref{formseri} are of order $|x_1|^{2j+2}$.

\subsection{Justification of Asymptotics}
%We present now the justification scheme for the first-order term
Let $\chi$ be a smooth cut-off function such that $0\leq\chi(x_1)\leq 1$ and 
      \begin{equation*}
      \label{cutoff}
      \chi(x_1) =0, \hspace{0.2cm}\text{if}\hspace{0.1cm} |x_1|\geq R_0, \qquad \chi (x_1) =1 \hspace{0.2cm}\text{if}\hspace{0.1cm} |x_1|\leq \frac{R_0}{2}.
      \end{equation*}
We set  
    \begin{equation}
    \label{asyform}
    u(x)=c_0 +\chi(x_1) \mathcal{U}_1(x)+ \remainder(x),
    \end{equation}
with $\remainder(x)$ being the remainder. 
      
\begin{comment}
We set 
     \begin{equation*}
     u(x) = u(\cusp) + \chi(x)\firstterm +\remainder(x).
     \end{equation*}
\end{comment}
%The main result is the following assertion.
   \begin{thm}
   \label{kissthm}
   The solution $u$ of the spectral problem \eqref{kisspb}-\eqref{kissbc} admits the asymptotic form \eqref{asyform}. More specifically, there exists an exponent $N>0$ such that the norm $$\|\rho(x)^{-N}\tilde{u}\|_{H^1_0(\ballone,\boundzero)}<\infty$$
   and the functions  $\rho(x)^{-N}(u(x)-c_0)$  and $\rho(x)^{-N} \firstterm $ do not belong to the Sobolev space $H^1$ in a neighbourhood of the cusp $\cusp$.
        % \begin{equation*}
         %u(x) = c_0 + \chi(x_1)\firstterm(x_1,\eta) +\tilde{u}(x),
         %\end{equation*}
   %where $\chi$ is the cut-off function defined in \eqref{cutoff} and $\tilde{u}$ is the remainder.
   \end{thm}     
 \begin{proof} 
 The remainder $\remainder$ verifies the following equation
    \begin{equation}
    \label{eqrem}
    -\laplace\remainder(x)-\lambda\remainder(x) = \lambda c_0+\lambda\mathcal{U}_1(x)\chi(x_1) +\laplace (\chi(x_1)\mathcal{U}_1(x)),\qquad x\in\ballone\setminus\cusp,
    \end{equation} 
  along  with homogeneous boundary conditions
      \begin{equation*}
      \normalderone \tilde{u}(x) = 0, \hspace{0.2cm}x\in\boundone\setminus\cusp , \qquad \tilde{u}(x)=0,\hspace{0.2cm}x\in\boundzero\setminus\cusp .
      \end{equation*}
  Multiplying \eqref{eqrem} by an arbitrary test function $v\in H^1_0(\ballone, \boundzero)$ and integrating in $\ballone$, we find 
     \begin{align}
     \label{probrem}
     (-\laplace\remainder, v)_{\ballone} -\lambda(\remainder,v)_{\ballone} &=(\laplace(\U_1\chi), v)_{\ballone}+\lambda(c_0, v)_{\ballone} +\lambda (\U_1\chi,v)_{\ballone} \notag \\
     &= (\chi\laplace\U_1, v)_{\ballone} + ([\laplace, \chi]\U_1,v)_{\ballone} +\lambda(c_0, v)_{\ballone}+ \lambda (\U_1\chi,v)_{\ballone}.
     \end{align}
 The commutator $[\laplace,\chi]$ is 
    \begin{equation*}
    [\laplace,\chi]\U = 2\grad \U\cdot\grad\chi + \U\laplace\chi.  
    \end{equation*}
 Let $\rho$ denote a smooth positive function on $\ballone$ which coincides with the distance to the origin of the Cartesian coordinate system in a neighbourhood of the cuspidal point $\cusp$ and introduce the weight function
      \begin{equation*}
      T_{\delta}(x) =
      \begin{cases}
      \delta^{-N}, & \quad\text{if}\quad\rho(x)\leq \delta,\\
      \rho(x)^{-N}, & \quad\text{if}\quad \delta<\rho(x)\leq R_0/2,\\
      (R_0/2)^{-N}, & \quad\text{if}\quad\rho(x)> R_0/2,\\
      \end{cases}
      \end{equation*}
 where the parameter $\delta>0$ is small and will be sent to $0$. Later on, we will impose some constraints on the exponent $N$. We point out that the derivative of $T_{\delta}$ vanishes for $\rho(x)\leq\delta$, $\rho(x)> R_0/2$ and satisfies the inequality
    \begin{equation}
    \label{estweight}
    %|T_{\delta}(x)|\leq\rho(x)^{-N} \qquad 
    |\grad T_{\delta}(x)|\leq CT_{\delta}(x)\rho(x)^{-1},\qquad \delta<\rho(x)\leq R_0.
    \end{equation}
 Since $\tilde{u}\in H^1_0(\ballone, \boundzero)$, we choose as a test function $V=T_{\delta}\tilde{v}\in H^1_0(\ballone, \boundzero) $, with $\tilde{v} = T_{\delta}\remainder$. After algebraic transformations, the left-hand side of \eqref{probrem} can be written as 
    \begin{align}
    -(\laplace \remainder, V)_{\ballone} - \lambda(\remainder, V)_{\ballone} &= (\grad\remainder, \grad V)_{\ballone} - \lambda (\tilde{v},\tilde{v})_{\ballone}\notag \\
    &= (\grad\remainder, \tilde{v}\grad T_{\delta})_{\ballone} + (\grad\remainder, T_{\delta}\grad\tilde{v})_{\ballone} - \lambda(\tilde{v},\tilde{v})_{\ballone} \notag\\
    &=(T_{\delta}\grad\remainder,\tilde{v}T_{\delta}^{-1}\grad T_{\delta})_{\ballone} + (T_{\delta}\grad\remainder, \grad\tilde{v})_{\ballone}- \lambda(\tilde{v},\tilde{v})_{\ballone} \notag\\
    &=(\grad\tilde{v},\tilde{v}T_{\delta}^{-1}\grad T_{\delta})_{\ballone}-(\remainder\grad T_{\delta},\tilde{v}T_{\delta}^{-1}\grad T_{\delta})_{\ballone} + (\grad\tilde{v}, \grad\tilde{v})_{\ballone}\notag\\
    &\quad - (\remainder\grad T_{\delta}, \grad\tilde{v})_{\ballone} - \lambda(\tilde{v},\tilde{v})_{\ballone} \notag\\
    &=(\grad\tilde{v}, \grad\tilde{v})_{\ballone} - (\remainder\grad T_{\delta},\tilde{v}T_{\delta}^{-1}\grad T_{\delta})_{\ballone}- \lambda(\tilde{v},\tilde{v})_{\ballone}\notag\\
    &= \|\grad\tilde{v}\|_{L^2(\ballone)}^2 - \|\tilde{v}T_{\delta}^{-1}\grad T_{\delta}\|_{L^2(\ballone)}^2 - \lambda\|\tilde{v}\|_{L^2(\ballone)}^2. \label{est1}
    \end{align}
 From formulas \eqref{probrem} and \eqref{est1}, we find that 
     \begin{align}
         \|\grad\tilde{v}\|_{L^2(\ballone)}^2 %&=
          %-(\laplace \remainder, V)_{\Omega} - \lambda(\remainder, V)_{\Omega} +\|\tilde{v}T_{\delta}^{-1}\grad T_{\delta}\|_{L^2(\Omega)}^2+\lambda\|\tilde{v}\|^2_{L^2(\Omega)}\notag\\
         &= (\chi\laplace\U_1, V)_{\ballone} + ([\laplace, \chi]\U_1,V)_{\ballone} +\lambda(c_0, V)_{\ballone}+ \lambda (\U_1\chi,V)_{\ballone}\notag\\
         &\quad+\|\tilde{v}T_{\delta}^{-1}\grad T_{\delta}\|_{L^2(\ballone)}^2+\lambda\|\tilde{v}\|^2_{L^2(\ballone)}.\label{gradtildev}
         %& + \lambda(c_0,v)_{\Pi} +([\laplace, \chi]\U,V)_{\Pi} + \lambda (\U\chi,V)_{\Pi}.
         \end{align}
 We estimate each terms in the previous equality. %Let $\Pi$ be the space $\{x\in\ballone\hspace{0.02cm}:\hspace{0.02cm} |x_1|\leq R_0\}$.  
 Bearing in mind that the correction term $\U_1$ is the solution of \eqref{pbfirst}, we obtain
      \begin{align*}
       (\chi\laplace\U_1, V)_{\ballone}=(\chi\frac{\partial^2}{\partial x_1^2}\U_1, V)_{\ballone} + (\chi\frac{\partial^2}{\partial x_2^2}\U_1, V)_{\ballone} = (\chi\frac{\partial^2}{\partial x_1^2}\U_1, V)_{\ballone} - (\chi\lambda c_0, V)_{\ballone}.
       \end{align*}
 Therefore,
     \begin{align*}
     |\lambda(c_0, V)_{\ballone}- (\chi\lambda c_0, V)_{\ballone}| \leq |\lambda(c_0, V)_{\ballone\cap\{x: \rho(x)\geq R_0/2\}}| \leq \lambda c_0 (R_0/2)^{-2N} \|\tilde{u}\|^2_{L^2(\ballone)}<\infty.
     \end{align*}
 Moreover, from Poincar\`{e}'s inequality 
      \begin{equation}
       \label{fried}
       \| H(R_0)^{-1}\tilde{v}(x)\|_{L^2(\ballone)} \leq C \| \grad \tilde{v}(x)\|_{L^2(\ballone)}
       %\int_{\Pi} H(R_0)^{-2}|\tilde{v}(x)|^2dx\leq C \int_{\Pi} \biggl|\frac{\partial \mathcal{U}}{\partial x_2}(x)  \biggr|^2 dx,
       \end{equation}
 we find
    \begin{align*}
    \left|(\chi\frac{\partial^2}{\partial x^2}\U_1, V)_{\ballone} \right|&= |(\chi  T_{\delta}\frac{\partial^2}{\partial x^2}\U_1, \tilde{v})_{\ballone}|= |(\chi  T_{\delta}H(x_1)\frac{\partial^2}{\partial x_1^2}\U_1, H(x_1)^{-1}\tilde{v})_{\ballone}|\\
    &\leq \|\chi  T_{\delta}H(x_1)\frac{\partial^2}{\partial x_1^2}\U_1\|_{L^2(\ballone
    )}\|H(x_1)^{-1}\tilde{v}\|_{L^2(\ballone)}\\
    &\leq C \|\chi  T_{\delta}H(x_1)\frac{\partial^2}{\partial x_1^2}\U_1\|_{L^2(\ballone)}\|\grad\tilde{v}\|_{L^2(\ballone)}.
    \end{align*}
 Taking into account that $H(x_1)=O(|x_1|^2)$ and $\frac{\partial^2}{\partial x_1^2}\U_1 = O(|x_1|^2)$, the norm 
      \begin{align*}
      \|\chi  T_{\delta}H(x_1)\frac{\partial^2}{\partial x_1^2}\U_1\|_{L^2(\ballone)} &\leq \|\rho^{-N}H(x_1)\frac{\partial^2}{\partial x_1^2}\U_1\|_{L^2(\ballone\cap\{x:\rho(x)\leq R_0/2\}    )} \\
      &\quad+ (R_0/2)^{-2N}\|H(x_1)\frac{\partial^2}{\partial x_1^2}\U_1\|_{L^2(\ballone\cap\{x:\rho(x)> R_0/2\}    )}
      \end{align*}
 is finite for $N<\frac{11}{5}$. The term $([\laplace, \chi]\U_1,V)_{\ballone}$ involves the derivatives of the cut-off function $\chi$. Then it does not vanish only if $ R_0/2<\rho(x)<R_0$ and
     \begin{align*}
     |([\laplace, \chi]\U_1,V)_{\ballone}| &= |([\laplace, \chi]\U_1,V)_{\{x\in\ballone\hspace{0.02cm}:\hspace{0.02cm} R_0/2<\rho(x)<R_0\}}| \\
     &=|(T_{\delta}[\laplace, \chi]\U_1,\tilde{v})_{\{x\in\ballone\hspace{0.02cm}:\hspace{0.02cm} R_0/2<\rho(x)<R_0\}}| \\
     &\leq \| T_{\delta}[\laplace, \chi]H(x_1)\U_1 \|_{L^2(\{x\in\ballone\hspace{0.02cm}:\hspace{0.02cm} R_0/2<\rho(x)<R_0\})}\\
     &\quad \times\|H^{-1}(x_1)\tilde{v}  \|_{L^2(\{x\in\ballone\hspace{0.02cm}:\hspace{0.02cm} R_0/2<\rho(x)<R_0\})}\\
     &\leq C \| T_{\delta}[\laplace, \chi]H(x_1)\U_1 \|_{L^2(\{x\in\ballone\hspace{0.02cm}:\hspace{0.02cm} R_0/2<\rho(x)<R_0\})}\\
     &\quad\times\|\grad\tilde{v}  \|_{L^2(\{x\in\ballone\hspace{0.02cm}:\hspace{0.02cm} R_0/2<\rho(x)<R_0\})},
     \end{align*}
 which is finite for any value of $N$ since $T_{\delta}(x) = (R_0/2)^{-N}$ if $\rho(x)>R_0/2$. %turns out to be very small since the commutator involves the derivatives of the cut-off function $\chi$. 
According to the inequality \eqref{estweight}, we have
      \begin{equation*}
       \|\tilde{v}T_{\delta}^{-1}\grad T_{\delta}\|_{L^2(\ballone)}^2\leq C\|\rho^{-1}\tilde{v}\|^2_{L^2(\ballone)}.
      \end{equation*}
 Choosing $R_0$ such that $\lambda\leq CH(R_0)^{-2}$, from \eqref{fried} we deduce
     \begin{align*}
      \lambda\|\tilde{v}\|_{L^2(\ballone)}^2 & \leq CH(R_0)^{-2} \|\tilde{v}\|_{L^2(\ballone)}^2 \leq C \|\grad\tilde{v}\|^2_{L^2(\ballone)}   %= \lambda\|\tilde{v}\|_{L^2(\ballone\setminus\Pi)}^2+ \lambda\|\tilde{v}\|_{L^2(\Pi)}^2\leq C+ \lambda\|\tilde{v}\|_{L^2(\Pi)}^2\\
      %&\leq C+\frac{CH(R_0)^{-2}}{3}\|\tilde{v}\|^2_{L^2(\Pi)}\leq C+\frac{1}{3}\|\grad\tilde{v}\|^2_{L^2(\Pi)} \leq C+\frac{C}{3}\|\grad\tilde{v}\|^2_{L^2(\ballone)}.
      \end{align*}
     % \begin{equation*}
      %\|\tilde{v}T_{\delta}^{-1}\grad T_{\delta}\|_{L^2(\Omega)}^2  \leq C\|\rho^{-1}\tilde{v}\|^2_{L^2(\Omega)}\leq \frac{C}{3}\|\grad\tilde{v}\|^2_{L^2(\Omega)}
      %\end{equation*}   
 Finally, we have 
    \begin{align*}
    |(\U_1\chi,V)_{\ballone}|&= |(T_{\delta}\U_1\chi,\tilde{v})_{\ballone}| = |(H(x_1)T_{\delta}\U_1\chi,H(x_1)^{-1}\tilde{v})_{\ballone} |\\
    &\leq \|H(x_1)T_{\delta}\U_1\chi\|_{L^2(\ballone)} \|H(x_1)^{-1}\tilde{v}\|_{L^2(\ballone)}\leq  C\|H(x_1)T_{\delta}\U_1\chi\|_{L^2(\ballone)} \|\grad\tilde{v}\|_{L^2(\ballone)}.
    \end{align*}
 The norm 
    \begin{align*}
     \|\chi H(x_1)T_{\delta}\U_1\|_{L^2(\ballone)} &\leq \|\rho^{-N}H(x_1)\U_1\|_{L^2(\ballone\cap\{x:\rho(x)\leq R_0/2\}    )} \\
     &\quad+ (R_0/2)^{-2N}\|H(x_1)\U_1\|_{L^2(\ballone\cap\{x:\rho(x)> R_0/2\}    )}
    \end{align*}
 is finite if and only if $N< 15/2$. Setting $N<11/2$, the relation \eqref{gradtildev} implies that % $ \|\grad\tilde{v}\|^2_{L^2(\ballone)}<\infty$. By applying again Poincarè's inequality, we find that 
 \begin{equation*}
 \|\rho^{-1}\tilde{v}\|^2_{L^2(\ballone)}+\|\grad\tilde{v}\|^2_{L^2(\ballone)}\leq C<\infty.
 \end{equation*}
 Since $T_{\delta}$ is monotone increasing as $\delta\rightarrow 0$, the limit of the last, bounded expression exists.
 
 %Note that 
   %\begin{align*}
   %\|\tilde{v}\|_{L^2(\Pi)} &= %\|H(x_1)H(x_1)^{-1}\tilde{v}\|_{L^2(\Pi)} \\
   %&\leq \|H(x_1)\|_{L^2(\Pi)} \|H(x_1)^{-1}\tilde{v}\|_{L^2(\Pi)} \\
   %&\leq C \|\grad\tilde{v}\|_{L^2(\Pi)}.
  % \end{align*}
 %In the last inequality we used the Friedrichs inequality. Then using the second inequality in \eqref{estweight}, the left-hand side of \eqref{remaest} exceed 
     %\begin{align*}
      %\|\grad\tilde{v}\|_{L^2(\Pi)}^2 - \|\tilde{v}T_{\delta}^{-1}\grad T_{\delta}\|_{L^2(\Pi)}^2 - \lambda\|\tilde{v}\|_{L^2(\Pi)}^2 &\geq \|\grad\tilde{v}\|_{L^2(\Pi)}^2 - \|\tilde{v}\rho^{-1}\|_{L^2(\Pi)}^2 - \lambda\|\tilde{v}\|_{L^2(\Pi)}^2 \\
      %&\geq C\|\grad\tilde{v}\|_{L^2(\Pi)}^2 - \|\tilde{v}\rho^{-1}\|_{L^2(\Pi)}^2 
     %\end{align*} 
\end{proof}
Since terms in the formal series \eqref{formseri} are polynomials in $x$, we deduce the smoothness of the solution $u$ to the problem \eqref{kisspb}-\eqref{kissbc}. In virtue of estimate of Theorem \ref{kissthm}, the estimate of the reimander in a weighted $H^2$ follows from standard local estimates near smooth parts of the boundary (cf. \cite{NT18}).

\subsection{The Dirichlet case}
If we replace the boundary condition \eqref{kissbc} on $\boundzero$ with a homogeneous Dirichlet condition, i.e. $c_0=0$, then all eigenfunctions $u$ of the problem
      \begin{align}
      -\laplace u(x) &=\lambda u(x),\hspace{0.3cm} x\in\Omega_1,  \label{Dir1}\\
      \normalderone u (x) = 0,\hspace{1.0cm} x\in\boundone\setminus\cusp, \hspace{0.5cm}&\hspace{0.9cm} 
      u(x)=0,\hspace{1.0cm}x\in\boundzero\setminus\cusp,\label{Dir3}
      \end{align}
decay exponentially as $x\rightarrow \cusp$. %The variational identity of \eqref{Dir1}-\eqref{Dir3} takes the form
 %  \begin{equation}
  % \label{varformtd}
  % (\grad u, \grad \varphi)_{\ballone} = \lambda (u,\varphi)_{\ballone}\qquad \forall \varphi\in H^1_0(\ballone, \boundzero).
  % \end{equation}
 %:= \{u\in H^1(\Pi)\hspace{0.01cm}:\hspace{0.01cm} u_{\big|\boundzero} =0\}$. 
   \begin{prop}
   The eigenfunction $u\in H^1_0(\ballone, \boundzero)$ of the problem \eqref{Dir1}-\eqref{Dir3} decays exponentially as $x\to\cusp$.
   %There exists $\beta>0$ such that the weighted norm $\||x|^{-2}e^{\frac{\beta}{|x|}}u\|_{\ballone}$ is finite for every eigenfunction $$.
   \end{prop}
   \begin{proof}
   %Outside any neighborhood of the cuspidal point $\cusp$, the eigenfunctions $u$ lie in all the space $H^k(\ballone\setminus\Pi)$, for all $k\in\mathbb{N}$. We, now, investigate the behaviour of $u$ in $\Pi$.\\ 
   Let $T_{\delta}$ be the weight function defined by
        \begin{equation*}
        T_{\delta}(x) =
        \begin{cases}
        e^{\frac{\beta}{\delta}}, & \qquad |x_1|\leq \delta,\\
        e^{\frac{\beta}{|x_1|}}, & \qquad \delta<|x_1|\leq R,\\
        e^{\frac{\beta}{R}}, & \qquad |x_1|>R.
        \end{cases} 
        \end{equation*}
   Here, the parameter $\delta$ is small, positive and it will be sent to $0$ and $\beta>0$. Note that $T_{\delta}$ is a continuous function such that 
           \begin{equation*}
           \label{weightfu}
           |\grad T_{\delta}(x)|\leq \beta|x_1|^{-2}T_{\delta}(x), \qquad e^{\frac{\beta}{R}}\leq T_{\delta}(x)\leq e^{\frac{\beta}{\delta}}.
           \end{equation*} 
   We insert into the integral identity \eqref{varide} the test function $v=T_{\delta}U\in H^1_0(\ballone, \boundzero)$, with $U=T_{\delta}u$, obtaining
        \begin{align*}
        \lambda(u, v)_{\ballone} &= %\lambda(T_{\delta}u, T_{\delta}u)_{\Omega} = \lambda (U,U)_{\Omega} =
         (\grad u, \grad v)_{\ballone} %= (\grad u, \grad T_{\delta}^2u)_{\ballone} \\
        = (\grad u, U\grad T_{\delta})_{\ballone} + (\grad u, T_{\delta} \grad U)_{\ballone} \\
        &=(T_{\delta}\grad u, T_{\delta}^{-1}U\grad T_{\delta})_{\ballone} + (T_{\delta}\grad u,  \grad U)_{\ballone}\\
        &= (\grad U, T_{\delta}^{-1}U\grad T_{\delta})_{\ballone} - (u\grad T_{\delta}, T_{\delta}^{-1}U\grad T_{\delta})_{\ballone}\\ &\quad+(\grad U,  \grad U)_{\ballone} - (u\grad T_{\delta},  \grad U)_{\ballone} \\
        &=(\grad U,  \grad U)_{\ballone} - (T_{\delta}^{-1}U\grad T_{\delta}, T_{\delta}^{-1}U\grad T_{\delta})_{\ballone}.
        \end{align*}
    %We stress that the first and the last term in the penultimate equality cancel out since the eigenfunctions are real-valued. In other words, we obtained 
    Hence, 
        \begin{equation*}
        \|\grad U\|^2_{L^2(\ballone)} = \lambda\|U\|^2_{L^2(\ballone)} +\|T_{\delta}^{-1}U\grad T_{\delta}\|^2_{L^2(\ballone)}.
        \end{equation*}
  Taking into account the Poincar\`{e}'s inequality \eqref{fried}, we find
        %\begin{equation*}
         %\|T_{\delta}^{-1}U\grad T_{\delta}\|^2_{\ballone} \leq \|\beta|x_1|^{-2}U\|^2_{\ballone},\qquad 
       %\|\grad U\|_{\Omega}^2\geq \|H(x_1)^{-1}U\|^2_{\Omega},
       %\end{equation*}
   %Therefore
        \begin{equation*}
        \label{Dirichcas}
         (c- \beta^2)\||x_1|^{-2}U\|^2_{L^2(\ballone)}\leq\lambda\|U\|^2_{L^2(\ballone)} \leq \lambda e^{\frac{2\beta}{\delta}}\|u\|^2_{L^2(\ballone)} <\infty.
        \end{equation*}
   %Taking into account the relation \eqref{Hfunc},
        %\begin{equation*}
        %H_{i} (x_1) = \frac{|x_1|^2}{2R_{i}} + O(|x_1|^4), \qquad i=0,1
        %\end{equation*}
   In particular, choosing $\beta$ such that $0\leq\beta^2< c$, we get
        \begin{equation*}
        e^{-\frac{2\beta}{\delta}}\int_{\ballone}|x_1|^{-4}|U(x)|^2dx \leq c_{\lambda}<\infty.
        \end{equation*}
   It implies that both of the integrals
       \begin{equation*}
       e^{-\frac{2\beta}{\delta}}\int_{\ballone\cap \{x: |x_1|\leq e^{-\frac{\beta}{2\delta}} \} }|x_1|^{-4}|U(x)|^2dx, \qquad e^{-\frac{2\beta}{\delta}}\int_{\ballone\cap \{x: |x_1|> e^{-\frac{\beta}{2\delta}} \} }|x_1|^{-4}|U(x)|^2dx,
       \end{equation*}
   are bounded for all $\delta>0$. The first one gives
       \begin{equation*}
       \int_{\ballone\cap \{x: |x_1|\leq e^{-\frac{\beta}{2\delta}} \} }|U(x)|^2dx \leq e^{-\frac{2\beta}{\delta}}\int_{\ballone\cap \{x: |x_1|\leq e^{-\frac{\beta}{2\delta}} \} }|x_1|^{-4}|U(x)|^2dx <\infty
       \end{equation*}
   for all $\delta>0$. Since $T_\delta$ is monotone increase as $\delta\to 0 $, we conclude that the eigenfunction $u$ has an exponential decay in $L^2$-norm in a neighbourhood of the cusp $\cusp$.
   \end{proof}
The eigenfunctions $u$ are thus smooth at any distance of $\cusp$ and vanish at cusp point $\cusp$ with all their derivatives due to the exponential decay. We conclude that also in this case the asymptotic anz\"{a}tze for $(\eigenv, u^\varepsilon)$ and the procedure given in the Sections $2-3$ are corect.

\subsection{Open Questions} 
Due to the shape of the boundary $\boundone$, the solution of the problem $\eqref{kisspb}-\eqref{kissbc}$ behaves in substantially different way from the solution of the problem 
     \begin{align}
     -\laplace u(x) &=\lambda u(x),\hspace{0.3cm} x\in\Omega_1,\label{laptouchdir}\\
      u (x) = 0,\hspace{0.3cm} x\in\boundone\setminus\cusp,\hspace{0.3cm}&\hspace{0.5cm}
      u(x)  =c_0,\hspace{0.3cm}x\in\boundzero\setminus\cusp.\label{laptouchdir1}
      \end{align}
Here we have simply replaced the Neumann boundary condition \eqref{Neumann} of the problem \eqref{kisspb}-\eqref{kissbc} with a homogeneous Dirichlet condition. Indeed an approximation of the solution $u$ of the problem \eqref{laptouchdir}-\eqref{laptouchdir1} is to be found in such a way that the boundary conditions are satisfied exactly while discrepancies in the equation \eqref{laptouchdir} is reduced as much as possible. As a consequence, a solution $u$ with the asymptotic 
   \begin{equation*}
   u(x) = c_0\frac{x_2-H_1(x_1)}{H(x_1)} +\cdots, \qquad x\to\cusp,
   \end{equation*}
cannot belong to the Sobolev space $H^1(\ballone)$. Indeed the integral 
       \begin{equation*}
       \int_{0}^{1/3}\int_{H_1(x_1)}^{H_0(x_1)}\left|\frac{\partial}{\partial x_2}\left(c_0\frac{x_2-H_1(x_1)}{H(x_1)}\right)\right|^2dx_2dx_1 = 2c_0^2\int_{0}^{1/3}\frac{1}{H(x_1)^2}dx_1
       \end{equation*}
is divergent because the integrand has nonadmissible singularity $O(|x_1|^{-4})$. The derivation of the ansatz for the eigenfunction $u$ of the problem \eqref{laptouchdir}-\eqref{laptouchdir1} is still an open problem.

\bigskip
\noindent
\section*{Acknowledgments}
\medskip
The author S.A. Nazarov has been supported by the grant 18-01-00325 of the Russian Foundation on Basic Research. The author S.A. Nazorv gratefully acknowledge the hospitality of the Department of Mathematical Sciences \textquotedblleft G.L. Lagrange\textquotedblright, Dipartimento di Eccellenza 2018-2022, of Politecnico di Torino, and the support of GNAMPA-INdAM. The authors V. Chiad\`o Piat  and L. D'Elia are members of the Gruppo Nazionale per l'Analisi Matematica, la Probabilit\`a e le loro Applicazioni (GNAMPA) of the Istituto Nazionale di Alta Matematica (INdAM).

\bigskip
\noindent

\bibliographystyle{plain}
\bibliography{kissingdomain}

\end{document}